\documentclass[12pt]{amsart}

\usepackage{amsmath}
\usepackage{amssymb}
\usepackage{array}

\renewcommand{\atop}[2]{%
\genfrac{}{}{0pt}{}{#1}{#2}}

\newtheorem{theorem}{Theorem}
\newtheorem{lemma}{Lemma}
\newtheorem{corollary}{Corollary}

\theoremstyle{definition}

\renewcommand{\atop}[2]{%
\genfrac{}{}{0pt}{}{#1}{#2}}

\begin{document}

\title{Bivariate identities for values of the Hurwitz zeta function and supercongruences}

\author{Kh.~Hessami Pilehrood$^{1}$}

\thanks{$^1$ This research was in part supported by a grant
from IPM (No. 89110024)}





\author{T.~Hessami Pilehrood$^2$}
\address{\begin{flushleft} School of Mathematics, Institute for Research  in Fundamental Sciences
(IPM), P.O.Box 19395-5746, Tehran, Iran \\ \hspace{1cm}  \\
Mathematics Department,  Faculty of Basic
Sciences,  Shahrekord University,  \\  P.O. Box 115, Shahrekord, Iran \end{flushleft}}

\email{hessamik@ipm.ir, hessamit@ipm.ir, hessamit@gmail.com}
\thanks{$^2$ This research was in part supported by a grant
from IPM (No. 89110025)}


\date{}


\maketitle

\centerline{\small\it Dedicated to D.~Zeilberger on the occasion of his 60th birthday}

\begin{abstract}
In this paper, we prove a new identity for values of the Hurwitz zeta function
 which contains as particular cases Koecher's identity for odd zeta values, the Bailey-Borwein-Bradley
identity for even zeta values and many other interesting formulas
related to values of the Hurwitz zeta function. We also get an extension of the bivariate identity of Cohen
 to values of the Hurwitz zeta function. The main tool we use here is a construction of new
Markov-WZ pairs. As application of our results, we prove several conjectures on supercongruences proposed by
J.~Guillera, W.~Zudilin, and Z.-W.~Sun.
\end{abstract}


\section{Introduction} \label{SS1}
\label{intro}

The Riemann zeta function for ${\rm Re}(s)>1$ is defined by the
series
$$
\zeta(s)=\sum_{n=1}^{\infty}\frac{1}{n^s}={}\sb {s+1}F\sb
s\left(\left.\atop{1,\ldots,1}{2,\ldots,2}\right|1\right),
$$
where
$$
{}\sb pF\sb q\left(\left.\atop{a_1,\ldots,
a_p}{b_1,\ldots,b_q}\right|z\right)=
\sum_{n=0}^{\infty}\frac{(a_1)_n\cdots (a_p)_n}{(b_1)_n\cdots
(b_q)_n}\frac{z^n}{n!}
$$
is the generalized hypergeometric function and $(a)_n$ is the
shifted factorial defined by $(a)_n=a(a+1)\cdots (a+n-1),$ $n\ge
1,$ and $(a)_0=1.$

In 1978, R.~Ap\'ery used the faster convergent series for $\zeta(3),$
\begin{equation}
\zeta(3)=\frac{5}{2}\sum_{k=1}^{\infty}\frac{(-1)^{k-1}}{k^3\binom{2k}{k}}
\label{eq01}
\end{equation}
to derive the irrationality of this number \cite{po}.
The series (\ref{eq01}) first obtained by A.~A.~Mar\-kov \cite{ma} in 1890 converges
exponentially faster than the original series for $\zeta(3),$ since by Stirling's formula,
$$
\frac{1}{k^3\binom{2k}{k}}\sim \frac{\sqrt{\pi}}{k^{5/2}} \, 4^{-k} \qquad\quad (k\to +\infty).
 $$
A general formula giving analogous Ap\'ery-like  series
for all $\zeta(2n+3),$ $n\ge 0,$ was proved by Koecher \cite{ko}
(and independently in an expanded form by  Leshchiner \cite{le}).
For $|a|<1,$ it reads
\begin{equation}
\sum_{n=0}^{\infty}\zeta(2n+3)a^{2n}=\sum_{k=1}^{\infty}\frac{1}{k(k^2-a^2)}=
\frac{1}{2}\sum_{k=1}^{\infty}\frac{(-1)^{k-1}}{k^3\binom{2k}{k}}\,\,\frac{5k^2-a^2}%
{k^2-a^2}\,\prod_{m=1}^{k-1}\left(1-\frac{a^2}{m^2}\right).
\label{eq02}
\end{equation}
Expanding the right-hand side of (\ref{eq02}) by powers of $a^2$ and
comparing coefficients of $a^{2n}$ on both sides leads to the
Ap\'ery-like series for $\zeta(2n+3)$ \cite{le}:
\begin{equation}
\zeta(2n+3)=\frac{5}{2}\sum_{k=1}^{\infty}\frac{(-1)^{k+1}}{k^3\binom{2k}{k}}(-1)^n e_n^{(2)}(k)
+2\sum_{j=1}^n\sum_{k=1}^{\infty}\frac{(-1)^{k+1}}{k^{2j+3}\binom{2k}{k}} (-1)^{n-j} e_{n-j}^{(2)}(k),
\label{eq03}
\end{equation}
where for positive  integers $r,s,$
$$
e_r^{(s)}(k):=[\,t^r] \prod_{j=1}^{k-1}(1+j^{-s}t)=
\sum_{1\le j_1<j_2<\ldots<j_r\le k-1} (j_1j_2\cdots j_r)^{-s},
$$
and $[\,t^r]$ means the coefficient of $t^r.$
In particular, substituting $n=0$ in (\ref{eq03})  recovers Markov's formula
(\ref{eq01})  and setting $n=1,2$ gives the
following two formulas:
\begin{eqnarray}
\qquad\zeta(5)& = & 2\sum_{k=1}^{\infty}\frac{(-1)^{k+1}}{k^5\binom{2k}{k}}-\frac{5}{2}\sum_{k=1}^{\infty}
\frac{(-1)^{k+1}}{k^3\binom{2k}{k}}\sum_{j=1}^{k-1}\frac{1}{j^2}, \label{zeta5}\\
\zeta(7)& = & 2\sum_{k=1}^{\infty}\frac{(-1)^{k+1}}{k^7\binom{2k}{k}}-2\sum_{k=1}^{\infty}
\frac{(-1)^{k+1}}{k^5\binom{2k}{k}}\sum_{j=1}^{k-1}\frac{1}{j^2}+\frac{5}{2}
\sum_{k=1}^{\infty}\frac{(-1)^{k+1}}{k^3\binom{2k}{k}}\sum_{m=1}^{k-1}\frac{1}{m^2}\sum_{j=1}^{m-1}
\frac{1}{j^2}, \label{zeta71}
\end{eqnarray}
respectively. In 1996, inspired by this result, J.~Borwein and D.~Bradley \cite{bb1} applied extensive
computer searches on the base of integer relations algorithms looking for additional zeta identities
of this sort. This led to the discovery of the new identity
\begin{equation}
\zeta(7)=\frac{5}{2}\sum_{k=1}^{\infty}\frac{(-1)^{k+1}}{k^7\binom{2k}{k}}+\frac{25}{2}
\sum_{k=1}^{\infty}\frac{(-1)^{k+1}}{k^3\binom{2k}{k}}\sum_{j=1}^{k-1}\frac{1}{j^4},
\label{zeta7}
\end{equation}
which is simpler than Koecher's formula for $\zeta(7),$
and similar identities for $\zeta(9),$ $\zeta(11),$ $\zeta(13),$ etc. This allowed them to conjecture
that certain of these identities, namely those  for $\zeta(4n+3)$ are given by
the following generating function formula \cite{bb}:
\begin{equation}
\sum_{n=0}^{\infty}\zeta(4n+3)a^{4n}=\sum_{k=1}^{\infty}\frac{k}{k^4-a^4}
=\frac{5}{2}\sum_{k=1}^{\infty}\frac{(-1)^{k+1} k}{\binom{2k}{k}(k^4-a^4)}
\prod_{m=1}^{k-1}\left(\frac{m^4+4a^4}{m^4-a^4}\right), \quad |a|<1.
\label{4k3}
\end{equation}
The validity of (\ref{4k3})  was proved later by G.~Almkvist and A.~Granville  \cite{algr} in 1999.
Expanding the right-hand side of (\ref{4k3}) in powers of $a^4$ gives the following Ap\'ery-like series
for $\zeta(4n+3)$ \cite{bb}:
\begin{equation}
\zeta(4n+3)=\frac{5}{2}\sum_{j=0}^n\sum_{k=1}^{\infty}\frac{(-1)^{k+1}}{k^{4j+3}\binom{2k}{k}}
\sum_{r=0}^{n-j}4^rh_{n-j-r}^{(4)}(k) e_r^{(4)}(k),
\label{eq04}
\end{equation}
where
$$
h_r^{(s)}(k):=[\,t^r] \prod_{j=1}^{k-1}(1-j^{-s}t)^{-1}.
$$
In particular, substituting $n=0$ in (\ref{eq04})  gives (\ref{eq01}) and
putting $n=1$ yields (\ref{zeta7}).
It is easily seen that
for $n\ge 1,$  formula (\ref{eq04}) contains fewer summations than  the corresponding formula
for $\zeta(4n+3)$ given by (\ref{eq03}).

There exists a bivariate unifying formula for identities (\ref{eq02}) and (\ref{4k3})
\begin{equation}
\sum_{k=1}^{\infty}\frac{k}{k^4-x^2k^2-y^4}=\frac{1}{2}\sum_{k=1}^{\infty}
\frac{(-1)^{k+1}}{k\binom{2k}{k}}\frac{5k^2-x^2}{k^4-x^2k^2-y^4}
\prod_{m=1}^{k-1}\frac{(m^2-x^2)^2+4y^4}{m^4-x^2m^2-y^4}.
\label{2n4m3}
\end{equation}
It was originally conjectured by H.~Cohen and then proved by D.~Bradley \cite{b}
and, independently, by T.~Rivoal \cite{ri}. Their proof consists of reduction of
(\ref{2n4m3}) to a finite non-trivial combinatorial identity which can be proved
on the basis of Almkvist and Granville's work \cite{algr}. Another proof of (\ref{2n4m3})
based on application of WZ pairs was given by the authors in \cite{he2}.
Since
\begin{equation}
\sum_{k=1}^{\infty}\frac{k}{k^4-x^2k^2-y^4}=\sum_{n=0}^{\infty}\sum_{m=0}^{\infty}
\binom{n+m}{n}\zeta(2n+4m+3)x^{2n}y^{4m}, \quad |x|^2+|y|^4<1,
\label{genf}
\end{equation}
the formula (\ref{2n4m3}) generates Ap\'ery-like series for all $\zeta(2n+4m+3),$
$n,m \ge 0,$ convergent at the geometric rate with ratio $1/4$ and contains, as particular cases,
both identities (\ref{eq02})  and  (\ref{4k3}). Indeed, setting $x=a$ and $y=0$ yields
Koecher's identity (\ref{eq02}), and setting $x=0,$ $y=a$ yields the Borwein-Bradley identity
(\ref{4k3}). Putting
\begin{equation}
a^2:=\frac{x^2+\sqrt{x^4+4y^4}}{2}, \qquad b^2:=\frac{x^2-\sqrt{x^4+4y^4}}{2},
\label{a2b2}
\end{equation}
we can rewrite (\ref{2n4m3}) in a more symmetrical way
\begin{equation}
\sum_{k=1}^{\infty}\frac{k}{(k^2-a^2)(k^2-b^2)}=\frac{1}{2}\sum_{n=1}^{\infty}
\frac{(-1)^{n-1}(5n^2-a^2-b^2)(1\pm a\pm b)_{n-1}}{n\binom{2n}{n}(1\pm a)_n(1\pm b)_n}.
\label{ab}
\end{equation}
Here and below $(u\pm v\pm w)$ means that the product contains the factors $u+v+w,$
$u+v-w,$ $u-v+w,$ $u-v-w.$

In \cite{he2}, the authors showed that the generating function (\ref{genf}) also  has
a much more rapidly convergent representation, namely
\begin{equation}
\sum_{k=1}^{\infty}\frac{k}{k^4-x^2k^2-y^4}=\frac{1}{2}\sum_{n=1}^{\infty}
\frac{(-1)^{n-1} r(n)}{n\binom{2n}{n}}
\frac{\prod_{m=1}^{n-1}((m^2-x^2)^2+4y^4)}{\prod_{m=n}^{2n}(m^4-x^2m^2-y^4)},
\label{fast}
\end{equation}
where
$$
r(n)=205n^6-160n^5+(32-62x^2)n^4+40x^2n^3+(x^4-8x^2-25y^4)n^2+10y^4n+y^4(x^2-2).
$$
The identity (\ref{fast}) produces accelerated series for all $\zeta(2n+4m+3),$ $n,m \ge ,$
convergent at the geometric rate with ratio $2^{-10}.$ In particular, if $x=y=0$
we get Amdeberhan-Zeilberger's series \cite{az} for $\zeta(3),$
\begin{equation}
\zeta(3)=\frac{1}{2}\sum_{n=1}^{\infty}
\frac{(-1)^{n-1}(205n^2-160n+32)}{n^5\binom{2n}{n}^5}.
\label{saz}
\end{equation}
It is worth pointing out that both identities (\ref{2n4m3}) and (\ref{fast}) were proved in \cite{he2}
by using the same Markov-WZ pair (see also \cite[p.~702]{he3} for the explicit expression), but with the
help of different summation formulas.

A more general form of the bivariate identity (\ref{2n4m3}) for the generating function
\begin{equation*}
\begin{split}
\sum_{n=0}^{\infty}\sum_{m=0}^{\infty}\binom{m+n}{n}&(A_0\zeta(2n+4m+4)
+B_0\zeta(2n+4m+3)+C_0\zeta(2n+4m+2))x^{2n}y^{4m}\\
&=\sum_{k=1}^{\infty}
\frac{A_0+B_0k+C_0k^2}{k^4-x^2k^2-y^4}, \qquad\qquad |x|^2+|y|^4<1,
\end{split}
\end{equation*}
where $A_0, B_0, C_0$ are arbitrary complex numbers, was proved in \cite{he2}
by means of the Markov-Wilf-Zelberger theory. More precisely, we have
\begin{equation}
\sum_{k=1}^{\infty}
\frac{A_0+B_0k+C_0k^2}{k^4-x^2k^2-y^4}=\sum_{n=1}^{\infty}
\frac{d_n}{\prod_{m=1}^n(m^4-x^2m^2-y^4)},
\label{general}
\end{equation}
where
\begin{equation*}
\begin{split}
d_n&=\frac{(-1)^{n-1}B_0(5n^2-x^2)}{2n\binom{2n}{n}}\prod_{m=1}^{n-1}
((m^2-x^2)^2+4y^4) \\[3pt]
&+\frac{(40n+10)L_n+(35n^5-35n^3x^2+4n(3x^4+10y^4))L_{n-1}}{4(5n^2-2x^2)}
\end{split}
\end{equation*}
and $L_n$ is a solution of a certain second order linear difference equation with
polynomial coefficients in $n$ and $x, y$ with the initial values $L_0=C_0,$
$L_1=(5-2x^2)A_0/15+(5x^2-1-4(x^4+6y^4))C_0/30.$ If we take $A_0=C_0=0,$ $B_0=1$ in (\ref{general}),
then $L_n=0$ for all $n \ge 0$ and we get the bivariate identity (\ref{2n4m3}).

First results related to generating function identities
for even zeta values belong to   Leshchiner \cite{le} who proved (in an expanded form)
that for $|a|<1,$
\begin{equation}
\sum_{n=0}^{\infty}\left(1-\frac{1}{2^{2n+1}}\right)\zeta(2n+2)a^{2n}=\sum_{n=1}^{\infty}
\frac{(-1)^{n-1}}{n^2-a^2}=\frac{1}{2}\sum_{k=1}^{\infty}\frac{1}{k^2\binom{2k}{k}}\,\,
\frac{3k^2+a^2}{k^2-a^2}\,\prod_{m=1}^{k-1}\left(1-\frac{a^2}{m^2}\right).
\label{eq05}
\end{equation}
Comparing  constant terms on both sides of (\ref{eq05}) yields
$$
\zeta(2)=3\sum_{k=1}^{\infty}\frac{1}{k^2\binom{2k}{k}}.
$$
In 2006,
 D.~Bailey, J.~Borwein and D.~Bradley  \cite{bbb} proved another identity
\begin{equation}
\sum_{n=0}^{\infty}\zeta(2n+2)a^{2n}=
\sum_{k=1}^{\infty}\frac{1}{k^2-a^2}=
3\sum_{k=1}^{\infty}\frac{1}{\binom{2k}{k}(k^2-a^2)}
\prod_{m=1}^{k-1}\left(\frac{m^2-4a^2}{m^2-a^2}\right).
\label{2n2}
\end{equation}
It generates similar Ap\'ery-like series for the numbers $\zeta(2n+2),$
which are not covered by Leshchiner's result (\ref{eq05}).
In the same paper \cite{bbb},  a  generating function producing fast
convergent series for the sequence $\zeta(2n+4),$ $n=0,1,2,\ldots,$ was found,
which for
$|a|<1,$  has the form
\begin{equation}
\frac{1}{2}\sum_{n=0}^{\infty}
\left(1-\frac{1}{2^{2n+3}}-\frac{3}{6^{2n+4}B_{2n+4}}\right)\zeta(2n+4)
a^{2n}=\sum_{k=1}^{\infty}\frac{1}{k^2\binom{2k}{k}(k^2-a^2)}
\prod_{m=1}^{k-1}\left(1-\frac{a^2}{m^2}\right), \label{eq06}
\end{equation}
here $B_{2n}\in {\mathbb Q}$ are the even indexed Bernoulli numbers generated by
$$
x\coth(x)=\sum_{n=0}^{\infty}B_{2n} \frac{(2x)^{2n}}{(2n)!}.
$$
It was shown that the left-hand side of (\ref{eq06}) represents a Maclaurin expansion of the function
$$
\frac{\pi a\csc(\pi a)+3\cos(\pi a/3)-4}{4a^4}.
$$
Comparing constant terms in (\ref{eq06}) implies that
$$
\zeta(4)=\frac{36}{17}\sum_{k=1}^{\infty}\frac{1}{k^4\binom{2k}{k}}.
$$
The identity (\ref{eq06}) gives a formula for $\zeta(2n+4)$ which for
$n\ge 0$ involves  fewer summations then the corresponding formula
generated by (\ref{eq05}). Note that a unifying formula for  identities generating even zeta values similar to
the bivariate formula (\ref{2n4m3}) for the odd cases is not known.

The Hurwitz zeta function defined by
$$
\zeta(s,v)=\sum_{k=0}^{\infty}\frac{1}{(k+v)^s}
$$
for $s\in {\mathbb C},$ ${\rm Re}\,s>1$ and $v\ne 0,-1,-2,\ldots$ is a generalization
of the Riemann zeta function $\zeta(s)=\zeta(s,1).$ In this paper, we prove a new identity for values
$\zeta(2n,v)$ which contains as particular cases Koecher's identity (\ref{eq02}), the Bailey-Borwein-Bradley
identity (\ref{2n2}), some special case of identity (\ref{2n4m3}) and many other interesting formulas
related to values of the Hurwitz zeta function. We also get extensions of identities (\ref{2n4m3}) and
(\ref{fast}) to values of the Hurwitz zeta function. The main tool we use here is a construction of new
Markov-WZ pairs. As application of our results, we prove several conjectures on supercongruences proposed by
J.~Guillera and W.~Zudilin \cite{gz}, and Z.-W.~Sun \cite{sun1, sun2}.

\section{Background} \label{SS2}

\vspace{0.3cm}

We start by recalling several definitions and known facts related to
the Markov-Wilf-Zeilberger theory (see \cite{ma, mo, moze}). A
function $H(n,k)$, in the integer variables $n$ and $k,$ is called
{\it hypergeometric} or {\it closed form (CF)} if the quotients
$$
\frac{H(n+1,k)}{H(n,k)} \qquad\mbox{and} \qquad
\frac{H(n,k+1)}{H(n,k)}
$$
are both rational functions of $n$ and $k.$ A hypergeometric
function that can be written as a ratio of products of factorials is
called {\it pure-hypergeometric.} A pair of CF functions $F(n,k)$
and $G(n,k)$ is called a {\it WZ pair} if
\begin{equation}
F(n+1,k)-F(n,k)=G(n,k+1)-G(n,k). \label{WZ}
\end{equation}
A {\it P-recursive} function is a function that satisfies a linear
recurrence relation with polynomial coefficients. If for a given
hypergeometric function $H(n,k),$  there exists a polynomial
$P(n,k)$ in $k$ of the form
$$
P(n,k)=a_0(n)+a_1(n)k+\cdots+a_L(n)k^L,
$$
for some non-negative integer $L,$ and P-recursive functions
$a_0(n), \ldots, a_L(n)$ such that
$$
F(n,k):=H(n,k)P(n,k)
$$
satisfies
(\ref{WZ}) with some function $G,$ then a pair $(F,G)$ is called a
{\it Markov-WZ pair} associated with the kernel $H(n,k)$ (MWZ pair
for short). We call $G(n,k)$ an {\it MWZ mate} of $F(n,k).$ If
$L=0,$ then $(F,G)$ is simply a WZ pair.

In 2005, M.~Mohammed \cite{mo} showed that for any
pure-hypergeometric kernel $H(n,k),$ there exists a non-negative
integer $L$  and a polynomial $P(n,k)$ as above such that
$F(n,k)=H(n,k)P(n,k)$ has an MWZ mate $G(n,k)=F(n,k)Q(n,k),$ where
$Q(n,k)$ is a ratio of two P-recursive functions. Paper \cite{moze}
is accompanied by the Maple package MarkovWZ which, for a given
$H(n,k)$ outputs the polynomial $P(n,k)$ and the $G(n,k)$ as above.

From relation (\ref{WZ}) we get the following summation formulas.

\noindent {\bf Proposition A.} \cite[Theorem 2(b)]{mo} {\it Let
$(F,G)$ be an MWZ pair. If $\lim\limits_{n\to\infty}F(n,k)=0$ for
every $k\ge 0,$ then
\begin{equation}
\sum_{k=0}^{\infty}F(0,k)-\lim_{k\to\infty}\sum_{n=0}^{\infty}G(n,k)=
\sum_{n=0}^{\infty}G(n,0), \label{f1}
\end{equation}
whenever both sides converge.}

\noindent {\bf Proposition B.} \cite[Cor.~2]{mo} {\it Let $(F,G)$ be
an MWZ pair. If $\lim\limits_{k\to\infty}
\sum\limits_{n=0}^{\infty}G(n,k)=0,$  then
\begin{equation}
\sum_{k=0}^{\infty}F(0,k)= \sum_{n=0}^{\infty}(F(n,n)+G(n,n+1)),
\label{f2}
\end{equation}
whenever both sides converge.}

Formulas (\ref{f1}), (\ref{f2}) with an appropriate choice of
MWZ pairs can be used to convert a given hypergeometric series into
a different rapidly converging one.

To ensure wider applications of WZ pairs for proving hypergeometric identities
we use an approach due to I.~Gessel \cite{ge} (see also \cite[\S 7.3, 7.4]{pwz}).
It is based on the fact that if we have a WZ pair $(F,G),$ then we can easily find other WZ
pairs by the following rules.

\noindent {\bf Proposition C.} \cite[Th.~3.1]{ge} {\it Let $(F,G)$ be
a WZ pair.

{\rm(i)} \quad For any complex numbers $\alpha$ and $\beta,$
$(F(n+\alpha,k+\beta), G(n+\alpha,k+\beta))$
is a WZ pair.

{\rm(ii)} \quad For any complex number $\gamma,$ $(\gamma F(n,k), \gamma G(n,k))$ is a WZ pair.

{\rm (iii)} \quad If $p(n,k)$ is a gamma product such that $p(n+1,k)=p(n,k+1)=p(n,k)$ for all $n$
and $k$ for which $p(n,k)$ is defined, then
$
(p(n,k)F(n,k), p(n,k)G(n,k))
$
is a WZ pair.

{\rm (iv)} \quad $(F(-n,k), -G(-n-1,k))$ is a WZ pair.

{\rm (v)} \quad $(F(n,-k), -G(n,-k+1))$ is a WZ pair.

{\rm (vi)} \quad $(G(k,n), F(k,n))$ is a WZ pair.
}

The WZ pairs obtained from $(F,G)$ by any combination of {\rm(i)--(v)} are called the {\it associates}
of $(F,G).$ The WZ pair of the form {\rm(vi)} and all its associates are called the {\it duals}
of $(F,G).$

\section{The identities} \label{SS3}

\vspace{0.3cm}

\begin{theorem} \label{t1}
Let $a, \alpha, \beta\in {\mathbb C},$ $|a|<1,$ $\alpha\ne\beta,$ and $\alpha\pm a,$
$\beta\pm a$ be distinct from $0, -1, -2,\ldots.$ Then we have
\begin{equation}
\begin{split}
&\qquad\qquad\qquad
\frac{1}{\beta-\alpha}\sum_{n=0}^{\infty}a^{2n}(\zeta(2n+2,\alpha)-\zeta(2n+2,\beta)) \\[3pt]
&=
\sum_{n=1}^{\infty}\frac{(-1)^{n-1}
(1+\alpha-\beta)_{n-1}(1+\beta-\alpha)_{n-1}(1+2a)_{n-1}(1-2a)_{n-1}}%
{(\alpha+a)_n(\alpha-a)_n(\beta+a)_n(\beta-a)_n} \\[3pt]
&\qquad\qquad\times\frac{(5n^2+3n(\alpha+\beta-2)+2(\alpha-1)(\beta-1)-2a^2)}{n\binom{2n}{n}}.
\label{eq07}
\end{split}
\end{equation}
\end{theorem}
\begin{proof}
By the definition of the Hurwitz zeta function, we have
\begin{equation}
\frac{1}{\beta-\alpha}\sum_{n=0}^{\infty}a^{2n}(\zeta(2n+2,\alpha)-\zeta(2n+2,\beta))
=\sum_{k=0}^{\infty}\frac{2k+\alpha+\beta}{((k+\alpha)^2-a^2)((k+\beta)^2-a^2)}.
\label{eq08}
\end{equation}
Now define a Markov kernel $H(n,k)$ by the formula
$$
H(n,k)=\frac{(\alpha+a)_k(\alpha-a)_k(\beta+a)_k(\beta-a)_k(n+2k+\alpha+\beta)}%
{(\alpha+a)_{n+k+1}(\alpha-a)_{n+k+1}(\beta+a)_{n+k+1}(\beta-a)_{n+k+1}}.
$$
Applying the Maple package Markov-WZ we get the associated WZ pair
$$
F(n,k)=H(n,k) \frac{(-1)^n(1+\alpha-\beta)_n(1+\beta-\alpha)_n(1+2a)_n(1-2a)_n}{\binom{2n}{n}},
$$
$$
G(n,k)=F(n,k) \frac{5n^2+n(3\alpha+3\beta+4)+2\alpha\beta-2a^2+(2k+1)(1+\alpha+\beta)+2k(k+3n)}%
{2(2n+1)(n+2k+\alpha+\beta)}.
$$
Now by Proposition A, we obtain
$$
\sum_{k=0}^{\infty}F(0,k)=\sum_{n=0}^{\infty}G(n,0),
$$
which implies (\ref{eq07}).
\end{proof}
Multiplying both sides of (\ref{eq07}) by $\beta-\alpha$ and letting $\beta$ tend to infinity
we get an extension of the Bailey-Borwein-Bradley identity to values of the Hurwitz zeta function:
\begin{equation}
\sum_{n=0}^{\infty}a^{2n}\zeta(2n+2,\alpha)=\sum_{n=0}^{\infty}
\frac{(3n+2\alpha-2)(1+2a)_{n-1}(1-2a)_{n-1}}{n\binom{2n}{n}(\alpha+a)_n(\alpha-a)_n}.
\label{eq09}
\end{equation}
Setting $\alpha=1$ in (\ref{eq09}) yields  the Bailey-Borwein-Bradley identity (\ref{2n2}).

Replacing $a$ by $a/2$ and $\alpha,$  $\beta$ by $1+a/2,$ $1-a/2,$ respectively, in (\ref{eq07}) and taking into
account (\ref{eq08}), we get Koecher's identity (\ref{eq02}).

Letting $\beta$ tend to $\alpha$ in (\ref{eq07}) and using the equality
$$
\frac{d}{dv}\zeta(s,v)=-s\zeta(s+1,v),
$$
we get the following.
\begin{corollary} \label{c1}
Let $a, \alpha\in {\mathbb C},$ $|a|<1,$ and $\alpha\pm a\ne 0,-1,-2, \ldots.$
Then
\begin{equation*}
\begin{split}
&\qquad\qquad\qquad
\sum_{n=0}^{\infty}(n+1)\zeta(2n+3,\alpha) a^{2n}=\sum_{k=0}^{\infty}\frac{k+\alpha}{((k+\alpha)^2-a^2)^2} \\[3pt]
&=\frac{1}{2}\sum_{n=1}^{\infty}\frac{(-1)^{n-1}(5n^2+6n(\alpha-1)+2(\alpha-1)^2-2a^2)}{n\binom{2n}{n}}
\frac{(n-1)!^2(1+2a)_{n-1}(1-2a)_{n-1}}{(\alpha+a)_n^2(\alpha-a)_n^2}.
\end{split}
\end{equation*}
\end{corollary}
Taking $\alpha=1$ in Corollary \ref{c1}, we get the following identity for odd zeta values.
\begin{corollary} \label{c2}
Let $a\in {\mathbb C},$ $|a|<1.$ Then
\begin{equation}
\sum_{n=1}^{\infty}n\zeta(2n+1)a^{2n-2}=\sum_{k=1}^{\infty}\frac{k}{(k^2-a^2)^2}
=\frac{1}{2}\sum_{n=1}^{\infty}\frac{(-1)^{n-1}(5n^2-a^2)}{n\binom{2n}{n}(n^2-a^2)^2}
\prod_{m=1}^{n-1}\frac{1-4a^2/m^2}{(1-a^2/m^2)^2}.
\label{eq10}
\end{equation}
\end{corollary}
Note that the right-hand side equality of (\ref{eq10}) also follows from the bivariate identity
(\ref{2n4m3}) or (\ref{ab}) as was shown by D.~Bradley (see \cite[Cor.~1]{b}).

It is clear that identity (\ref{eq10}) gives formulas for odd zeta values which are linear combinations
of series generated by the bivariate identity (\ref{2n4m3}).
Thus comparing constant terms on both sides of (\ref{eq10}) gives Ap\'ery's
series (\ref{eq01}) for $\zeta(3).$ Similarly, comparing coefficients of $a^2$ gives formula (\ref{zeta5})
for $\zeta(5).$ It produces the following  complicated expression for $\zeta(7):$
\begin{equation*}
\begin{split}
\zeta(7)&=\frac{11}{6}\sum_{k=1}^{\infty}\frac{(-1)^{k-1}}{k^7\binom{2k}{k}}-\frac{8}{3}\sum_{k=1}^{\infty}
\frac{(-1)^{k-1}}{k^5\binom{2k}{k}}\sum_{j=1}^{k-1}\frac{1}{j^2}-\frac{25}{6}\sum_{k=1}^{\infty}
\frac{(-1)^{k-1}}{k^3\binom{2k}{k}}\sum_{j=1}^{k-1}\frac{1}{j^4} \\
&+\frac{10}{3}\sum_{k=1}^{\infty}
\frac{(-1)^{k-1}}{k^3\binom{2k}{k}}\sum_{j=1}^{k-1}\frac{1}{j^2}\sum_{m=1}^{j-1}\frac{1}{m^2},
\end{split}
\end{equation*}
which can be written as
$$
\zeta(7)=\frac{1}{3}(4K-B),
$$
where $K$ and  $B$ are right-hand sides of formulas (\ref{zeta71}) and (\ref{zeta7}), respectively.
More generally, if we denote
$$
g_r^{(s)}(k):=[\,t^r]\prod_{j=1}^{k-1}(1-j^{-s}t)^{-2},
$$
then taking into account that
$$
\frac{5k^2-2a^2}{(k^2-a^2)^2}=\frac{2}{1-a^2/k^2}+\frac{3}{(1-a^2/k^2)^2}=\sum_{j=0}^{\infty}
(3j+5)\frac{a^{2j}}{k^{2j}}
$$
and comparing the coefficients of $a^{2n}$ on both sides of (\ref{eq10}), we get
\begin{corollary} \label{c3}
Let $n$ be a non-negative integer. Then
\begin{equation}
\zeta(2n+3)=\frac{1}{2n+2}\sum_{j=0}^n(3j+5)\sum_{k=1}^{\infty}\frac{(-1)^{k-1}}{k^{2j+3}\binom{2k}{k}}
\sum_{r=0}^{n-j}(-4)^re_r^{(2)}(k) g_{n-j-r}^{(2)}(k).
\label{eq11}
\end{equation}
\end{corollary}
Consider several other particular cases of Theorem \ref{t1}. Replacing $a$ by $a/2,$ $\alpha$ by $1/2,$
and $\beta$ by $1$ in (\ref{eq07}) and noting that
$$
\sum_{k=0}^{\infty}\frac{2k+3/2}{((k+1/2)^2-a^2/4)((k+1)^2-a^2/4)}=8\sum_{n=1}^{\infty}
\frac{(-1)^{n-1}}{n^2-a^2},
$$
we get the following identity.

\begin{corollary} \label{c4}
Let $a$ be a complex number, distinct from a non-zero integer. Then
$$
\sum_{n=1}^{\infty}\frac{(-1)^{n-1}}{n^2-a^2}=\frac{1}{4}\sum_{n=1}^{\infty}\frac{(-1)^{n-1}(10n^2-3n-a^2)}%
{n(2n-1)(n^2-a^2)\binom{2n}{n}\prod_{j=1}^n(1-a^2/(n+j)^2)}.
$$
In particular,
$$
\zeta(2)=\frac{1}{2}\sum_{n=1}^{\infty}\frac{(-1)^{n-1}(10n-3)}{n^2(2n-1)\binom{2n}{n}}.
$$
\end{corollary}
Substituting  $\alpha=1/3,$ $\beta=2/3,$ $a=0$ in Theorem \ref{t1}, we get
$$
\zeta(2,1/3)-\zeta(2,2/3)=\frac{1}{3}\sum_{n=1}^{\infty}
\frac{(-1)^{n-1} n!^2 (15n-4)}{n^3\binom{2n}{n}(1/3)_n(2/3)_n}.
$$
Now observing that
$$
\left(\frac{1}{3}\right)_n\left(\frac{2}{3}\right)_n=\frac{(3n)!}{27^n n!}
$$
and
$$
\zeta(2,1/3)-\zeta(2,2/3)=9\sum_{n=1}^{\infty}\frac{(\frac{n}{3})}{n^2}=:9K
$$
(where $(\frac{n}{p})$ is the Legendre symbol), we get the following formula
$$
K=\sum_{n=1}^{\infty}\frac{(15n-4)(-27)^{n-1}}{n^3\binom{2n}{n}^2\binom{3n}{n}},
$$
which was conjectured by Z.-W.~Sun in \cite{sun1}.

Substituting $\alpha=1/4,$ $\beta=3/4$ recovers Theorem 3 from \cite{he2.5}
and in particular (when $a=0$), it gives the following formula for Catalan's constant
$G:=\sum_{n=0}^{\infty}\frac{(-1)^n}{(2n+1)^2}:$
$$
G=\frac{1}{64}\sum_{n=1}^{\infty}\frac{(-1)^{n-1}256^n(40n^2-24n+3)}{\binom{4n}{2n}^2\binom{2n}{n}n^3(2n-1)}.
$$
Applying Proposition B to the Markov-WZ pair found in the proof of Theorem \ref{t1}, we get the following
identity which generates Ap\'ery-like series for the differences $\zeta(2n+2,\alpha)-\zeta(2n+2,\beta)$
converging exponentially fast as $2^{-10}.$
\begin{theorem} \label{t2}
Let $a, \alpha, \beta\in {\mathbb C},$ $|a|<1,$ $\alpha\ne \beta,$ and $\alpha\pm a,$ $\beta\pm a$
be distinct from $0,-1,-2,\ldots.$ Then
\begin{equation}
\begin{split}
&\qquad\qquad\quad\frac{1}{\beta-\alpha}\sum_{n=0}^{\infty}(\zeta(2n+2,\alpha)-\zeta(2n+2,\beta))a^{2n} \\
&=\sum_{n=1}^{\infty}\frac{(-1)^{n-1} p_{\alpha,\beta}(n)}{n\binom{2n}{n}}
\frac{\prod_{j=1}^{n-1}(j^2-(\alpha-\beta)^2)(j^2-4a^2)}{\prod_{j=n-1}^{2n-1}((j+\alpha)^2-a^2)((j+\beta)^2-a^2)},
\label{eq12}
\end{split}
\end{equation}
where
\begin{equation*}
\begin{split}
&p_{\alpha,\beta}(n)=2(2n-1)(3n+\alpha+\beta-3)((2n-1+\alpha)^2-a^2)((2n-1+\beta)^2-a^2) \\
&+
((n+\alpha-1)^2-a^2)((n+\beta-1)^2-a^2)(13n^2+5n(\alpha+\beta-2)+2((1-\alpha)(1-\beta)-a^2)).
\end{split}
\end{equation*}
\end{theorem}
Setting $\alpha=1/3,$ $\beta=2/3,$ $a=0$ in Theorem \ref{t2}, we get the following fast converging series
for the constant $K:$
$$
K=\sum_{n=1}^{\infty}\frac{(-27)^{n-1}(5535n^3-4689n^2+1110n-80)}{n^3(3n-1)(3n-2)\binom{6n}{3n}^2\binom{3n}{n}}.
$$ 
Setting $\alpha=1/4,$ $\beta=3/4,$ we recover Theorem 4 from \cite{he2.5}.
 
Multiplying both sides of (\ref{eq12}) by $\beta-\alpha$ and letting $\beta$ tend to infinity
we obtain an extension of Theorem 2 from \cite{he1} to values of the Hurwitz zeta function
\begin{equation}
\sum_{n=0}^{\infty}\zeta(2n+2,\alpha)a^{2n}=\sum_{n=1}^{\infty}\frac{p(n)}{n\binom{2n}{n}}
\frac{\prod_{m=1}^{n-1}(m^2-4a^2)}{\prod_{m=n-1}^{2n-1}((m+\alpha)^2-a^2)},
\label{dobavl}
\end{equation}
where $p(n)=2(2n-1)((2n-1+\alpha)^2-a^2)+(5n+2\alpha-2)((n+\alpha-1)^2-a^2).$
In particular, setting $\alpha=1,$ $a=0$ in (\ref{dobavl}) we get Zeilberger's series \cite[\S 12]{zeil}
for $\zeta(2),$
$$
\zeta(2)=\sum_{n=1}^{\infty}\frac{21n-8}{n^3\binom{2n}{n}^3}.
$$
Replacing $a$ by $a/2$ and $\alpha, \beta$ by $1+a/2,$ $1-a/2,$ respectively, we recover Theorem 4 from
\cite{he1}.
Letting $\beta$ tend to $\alpha$ in (\ref{eq12}), we get the following
\begin{corollary} \label{c5}
Let $a, \alpha\in {\mathbb C},$ $|a|<1,$ and $\alpha\pm a$
be distinct from $0,-1,-2,\ldots.$ Then
\begin{equation*}
\begin{split}
&\qquad\qquad\qquad\qquad\qquad\qquad\qquad\sum_{k=1}^{\infty}k\zeta(2k+1,\alpha)a^{2k-2} \\
&=\frac{1}{2}\sum_{n=1}^{\infty}
\frac{(-1)^{n-1}p_{\alpha}(n)}{n\binom{2n}{n}^5((n+\alpha-1)^2-a^2)^2((n+\alpha-1)^2-a^2/4)^2}
\prod_{m=1}^{n-1}\frac{1-4a^2/m^2}{((1+\frac{\alpha-1}{m+n})^2-\frac{a^2}{(m+n)^2})^2},
\end{split}
\end{equation*}
where
\begin{equation*}
\begin{split}
p_{\alpha}(n):=p_{\alpha,\alpha}(n)&=2(2n-1)(3n+2\alpha-3)((2n+\alpha-1)^2-a^2)^2 \\
&+
((n+\alpha-1)^2-a^2)^2(13n^2+10n(\alpha-1)+2((1-\alpha)^2-a^2)).
\end{split}
\end{equation*}
\end{corollary}
Setting $\alpha=1$ in Corollary \ref{c5} we get the following identity.
\begin{corollary} \label{c6}
Let $a\in {\mathbb C},$ $|a|<1.$  Then
$$
\sum_{k=1}^{\infty}k\zeta(2k+1)a^{2k-2}=\frac{1}{2}\sum_{n=1}^{\infty}
\frac{(-1)^{n-1}p(n)}{n\binom{2n}{n}^5(n^2-a^2)^2(n^2-a^2/4)^2}
\prod_{m=1}^{n-1}\frac{1-4a^2/m^2}{(1-a^2/(m+n)^2)^2},
$$
where $p(n)=2(2n-1)(3n-1)(4n^2-a^2)^2+(n^2-a^2)^2(13n^2-2a^2).$
\end{corollary}
Setting $a=0$ in Corollary \ref{c6} we get Amdeberhan-Zeilberger's series (\ref{saz}) for $\zeta(3).$

The next theorem gives a generalization of identity (\ref{ab}).
\begin{theorem} \label{t3}
Let  $\alpha, a, b\in {\mathbb C}$  and $\alpha\pm a,$ $\alpha\pm b$
be distinct from $0,-1,-2,\ldots.$ Then the following identity holds:
\begin{equation}
\begin{split}
&\qquad\qquad\qquad\qquad\sum_{k=0}^{\infty}\frac{k+\alpha}{((k+\alpha)^2-a^2)((k+\alpha)^2-b^2)} \\
&=\frac{1}{2}
\sum_{n=1}^{\infty}\frac{(-1)^{n-1}(1\pm a\pm b)_{n-1}(5n^2-6n(1-\alpha)+2(1-\alpha)^2-a^2-b^2)}{n\binom{2n}{n}
(\alpha\pm a)_n(\alpha\pm b)_n}.
\label{eq13}
\end{split}
\end{equation}
\end{theorem}
\begin{proof}
Taking the kernel
$$
H(n,k)=\frac{(\alpha+a)_k(\alpha-a)_k(\alpha+b)_k(\alpha-b)_k(n+2k+2\alpha)}{(\alpha+a)_{n+k+1}
(\alpha-a)_{n+k+1}(\alpha+b)_{n+k+1}(\alpha-b)_{n+k+1}}
$$
and applying the Maple package MarkovWZ we get that
$$
F(n,k)=\frac{(-1)^n}{\binom{2n}{n}}(1\pm a\pm b)_n H(n,k)
$$
and
$$
G(n,k)=F(n,k)\frac{5n^2+6\alpha n+4n+2\alpha+2\alpha^2+1-a^2-b^2+k(2k+6n+4\alpha+2)}{2(2n+1)(n+2k+2\alpha)}
$$
give a WZ pair, i.e.,
$$
F(n+1,k)-F(n,k)=G(n,k+1)-G(n,k).
$$
Now by Proposition A, we get
$$
\sum_{k=0}^{\infty}F(0,k)=\sum_{n=0}^{\infty}G(n,0),
$$
which implies (\ref{eq13}).
\end{proof}
Making the substitution (\ref{a2b2}) in (\ref{eq13}) we get a generalization of Cohen's identity
to values of the Hurwitz zeta function.
\begin{corollary} \label{c7}
Let $x, y, \alpha\in {\mathbb C},$ $|x|^2+|y|^4<1$ and $\alpha\ne 0,-1,-2,\ldots.$ Then
\begin{equation*}
\begin{split}
&\sum_{k=0}^{\infty}\frac{k+\alpha}{(k+\alpha)^4-x^2(k+\alpha)^2-y^4}=\sum_{n=0}^{\infty}\sum_{m=0}^{\infty}
\binom{n+m}{n}\zeta(2n+4m+3,\alpha)x^{2n}y^{4m} \\
&=\!\frac{1}{2}\!\sum_{n=1}^{\infty}
\frac{(-1)^{n-1}(5n^2-6n(1-\alpha)+2(1-\alpha)^2-x^2)}{n\binom{2n}{n}((n+\alpha-1)^4-x^2(n+\alpha-1)^2-y^4)}
\prod_{j=1}^{n-1}\frac{(j^2-x^2)^2+4y^4}{(j+\alpha-1)^4-x^2(j+\alpha-1)^2-y^4}.
\end{split}
\end{equation*}
\end{corollary}
Setting $\alpha=1/2,$ $x=y=0$ in Corollary \ref{c7} we get the following formula:
$$
\zeta(3)=\frac{1}{28}\sum_{n=1}^{\infty}\frac{(-1)^{n-1}(10n^2-6n+1) 256^n}{n^5\binom{2n}{n}^5}.
$$
Applying Proposition B to the Markov-WZ pair used in the proof of Theorem \ref{t3} we get the following identity.
\begin{theorem} \label{t4}
Let $\alpha, a, b\in{\mathbb C}$  and $\alpha\pm a,$ $\alpha\pm b\ne 0,-1,-2,\ldots.$
Then
\begin{equation*}
\begin{split}
&\qquad\qquad\sum_{k=0}^{\infty}\frac{k+\alpha}{((k+\alpha)^2-a^2)((k+\alpha)^2-b^2)} \\
&=\frac{1}{2}
\sum_{n=1}^{\infty}\frac{(-1)^{n-1}(1\pm a\pm b)_{n-1}(\alpha\pm a)_{n-1}(\alpha\pm b)_{n-1}}{n\binom{2n}{n}
(\alpha\pm a)_{2n}(\alpha\pm b)_{2n}}\,q(n),
\end{split}
\end{equation*}
where
\begin{equation*}
\begin{split}
q(n)&=2(2n-1)(3n+2\alpha-3)((2n+\alpha-1)^2-a^2)((2n+\alpha-1)^2-b^2) \\
&+((n+\alpha-1)^2-a^2)((n+\alpha-1)^2-b^2)(13n^2
-10n(1-\alpha)+2(1-\alpha)^2-a^2-b^2).
\end{split}
\end{equation*}
\end{theorem}
Making the change of variables (\ref{a2b2}) in Theorem \ref{t4} we get a generalization of the identity (\ref{fast})
to values of the Hurwitz zeta function.
\begin{corollary} \label{c8}
Let $x, y, \alpha\in {\mathbb C},$ $|x|^2+|y|^4<1$ and $\alpha\ne 0,-1,-2,\ldots.$ Then
\begin{equation}
\begin{split}
&\qquad\quad\qquad\qquad\qquad\sum_{n=0}^{\infty}\sum_{m=0}^{\infty}\zeta(2n+4m+3,\alpha)x^{2n}y^{4m} \\
&=\frac{1}{2}\sum_{n=1}^{\infty}
\frac{(-1)^{n-1} Q(n)}{n\binom{2n}{n}}
\frac{\prod_{j=1}^{n-1}((j^2-x^2)^2+4y^4)}{\prod_{j=0}^n((n+j+\alpha-1)^4-x^2(n+j+\alpha-1)^2-y^4)},
\label{eq14}
\end{split}
\end{equation}
where
\begin{equation*}
\begin{split}
Q(n)&=2(2n-1)(3n+2\alpha-3)((2n+\alpha-1)^4-x^2(2n+\alpha-1)^2-y^4) \\
&+((n+\alpha-1)^4-x^2(n+\alpha-1)^2-y^4)(13n^2
-10n(1-\alpha)+2(1-\alpha)^2-x^2).
\end{split}
\end{equation*}
\end{corollary}
Formula  (\ref{eq14}) produces accelerated series for the values $\zeta(2n+4m+3,\alpha),$ $n,m\ge 0,$
convergent at the geometric rate with ratio $2^{-10}.$

{\bf Remark.} Note that Theorem \ref{t4} can also be obtained from Proposition A applied to the Markov-WZ pair
associated with the kernel
\begin{equation}
H_1(n,k)=\frac{(\alpha+a)_{k+n}(\alpha-a)_{k+n}(\alpha+b)_{k+n}(\alpha-b)_{k+n}}%
{(\alpha+a)_{2n+k+1}(\alpha-a)_{2n+k+1}(\alpha+b)_{2n+k+1}(\alpha-b)_{2n+k+1}}\,(3n+2k+2\alpha).
\label{H1nk}
\end{equation}

\vspace{0.3cm}

\section{Hypergeometric reformulations} \label{SS4}

\vspace{0.1cm}

Theorems \ref{t1}, \ref{t3} admit nice reformulations in terms of the hypergeometric and digamma functions
which  can be useful in their own.
Let us recall the definition of the digamma function that is the logarithmic derivative of the Gamma function
$$
\psi(z)=\frac{d}{dz}\log\Gamma(z)=-\gamma+\sum_{n=0}^{\infty}\left(\frac{1}{n+1}-\frac{1}{n+z}\right).
$$
Obviously, $\psi(1)=-\gamma,$ where $\gamma$ is Euler's constant. The function $\psi(z)$ is single-valued and
analytic in the whole complex plane except for the points $z=-m,$ $m=0,1,2,\ldots,$ where it has simple poles.
It's connection with the Hurwitz zeta function is given by the formula
$$
\psi^{(k)}(z)=\frac{d^k}{dz^k}\psi(z)=(-1)^{k+1} k!\zeta(k+1,z).
$$

Now it is easily seen that Theorems \ref{t1} and \ref{t3} are equivalent to the following statements.
\begin{theorem} \label{t5}
Let $a, \alpha, \beta$ be complex numbers such that $\alpha\pm a,$ $\beta\pm a$ be distinct from $0,-1, -2,\ldots.$
Then we have
\begin{equation*}
\begin{split}
\psi(\beta+a)&-\psi(\alpha+a)+\psi(\alpha-a)-\psi(\beta-a)=
\frac{a(\alpha-\beta)((1+\alpha)(1+\beta)+\alpha\beta-2a^2)}{(\alpha^2-a^2)(\beta^2-a^2)} \\[5pt]
&\times
{}\sb 8F\sb 7\left(\left.\atop{1, 1, 1+\alpha-\beta, 1+\beta-\alpha, 1\pm 2a, \frac{7}{5}+
\frac{3(\alpha+\beta)}{10}\pm \frac{\sqrt{D}}{10}
}{\frac{3}{2}, 1+\alpha\pm a, 1+\beta\pm a, \frac{2}{5}+\frac{3(\alpha+\beta)}{10}\pm\frac{\sqrt{D}}{10}}\right|-\frac{1}{4}\right),
\end{split}
\end{equation*}
where $D=40a^2+9(\alpha-\beta)^2-4(\alpha-1)(\beta-1).$
\end{theorem}
\begin{theorem} \label{t6}
Let $\alpha,  a, b$ be complex numbers such that $\alpha\pm a,$ $\alpha\pm b$ be distinct from $0,-1, -2,\ldots.$
Then we have
\begin{equation*}
\begin{split}
\psi(\alpha+b)&+\psi(\alpha-b)-\psi(\alpha+a)-\psi(\alpha-a)=
\frac{(a^2-b^2)((1+\alpha)^2+\alpha^2-a^2-b^2)}{2(\alpha^2-a^2)(\alpha^2-b^2)} \\[5pt]
&\times
{}\sb 8F\sb 7\left(\left.\atop{1,  1,  1\pm a\pm b, \frac{7+3\alpha}{5}\pm \sqrt{\frac{a^2+b^2}{5}-(\frac{1-\alpha}{5})^2}
}{\frac{3}{2}, 1+\alpha\pm a, 1+\alpha\pm b, \frac{2+3\alpha}{5}\pm\sqrt{\frac{a^2+b^2}{5}-(\frac{1-\alpha}{5})^2}}\right|-\frac{1}{4}\right).
\end{split}
\end{equation*}
\end{theorem}
Dividing both sides of the identity in Theorem \ref{t5} by $\beta-\alpha$ and letting $\beta$ tend to $\alpha,$
we get the following.
\begin{corollary} \label{c9}
Let $a, \alpha$ be complex numbers such that $\alpha\pm a$  be distinct from $0,-1, -2,\ldots.$
Then we have
\begin{equation*}
\begin{split}
\psi'(\alpha-a)-\psi'(\alpha+a)&=\frac{a((1+\alpha)^2+\alpha^2-2a^2)}{(\alpha^2-a^2)^2} \\
&\times
{}\sb 8F\sb 7\left(\left.\atop{1, 1, 1, 1, 1\pm 2a, \frac{7+3\alpha}{5}
\pm \sqrt{\frac{2a^2}{5}-(\frac{1-\alpha}{5})^2}
}{\frac{3}{2}, 1+\alpha\pm a, 1+\alpha\pm a, \frac{2+3\alpha}{5}\pm\sqrt{\frac{2a^2}{5}-(\frac{1-\alpha}{5})^2}}\right|-\frac{1}{4}\right).
\end{split}
\end{equation*}
\end{corollary}
Multiplying both sides of the above equality by $(\alpha-a)^2$ and letting $\alpha$ tend to $a,$
we get the following identity.
\begin{corollary} \label{c10}
Let $a$ be a complex number such that $2a$  be distinct from $0,-1, -2,\ldots.$
Then we have
\begin{equation*}
{}\sb 5F\sb 4\left(\left.\atop{1, 1, 1-2a, \frac{7+3a\pm\sqrt{9a^2+2a-1}}{5}
}
{\frac{3}{2}, 1+2a, \frac{2+3a\pm\sqrt{9a^2+2a-1}}{5}}\right|-\frac{1}{4}\right)=\frac{4a}{1+2a}.
\end{equation*}
\end{corollary}
Taking, for example, $a=1/4$ in the above identity we get the following non-trivial summation formula:
$$
\sum_{k=1}^{\infty}\frac{(-1)^{k-1} (5k-2)}{k\binom{2k}{k}}=1.
$$
\begin{corollary} \label{c11}
Let $n$ be a non-negative integer. Suppose that $\alpha, \beta$ are complex numbers such that
$2\beta$ and $\alpha+\beta$ are distinct from  $-n, -n-1, -n-2,\ldots,$ and $\alpha-\beta\ne n, n-1, n-2, \ldots.$
Then we have
\begin{equation}
\begin{split}
{}\sb 7F\sb 6\left(\left.\atop{n+1, n+1\pm(\alpha-\beta),  n+1\pm(2n+2\beta),
n+\frac{7}{5}+\frac{3(\alpha+\beta)}{10}
\pm \frac{\sqrt{D_1}}{10}
}{\frac{3}{2}+n, n+\alpha+1\pm (\beta+n), 2n+2\beta+1, n+\frac{2}{5} +\frac{3(\alpha+\beta)}{10}\pm\frac{\sqrt{D_1}}{10}}\right|-\frac{1}{4}\right) \\[5pt]
=\frac{(n+\alpha+\beta)_{n+1}(n+2)_{n+1}}{(1+\alpha-\beta)_{n+1}(1+2n+2\beta)_{n+1}},
\label{eq15}
\end{split}
\end{equation}
where $D_1=40((n+\beta)^2-(\alpha-1)(\beta-1))+9(\alpha+\beta-2)^2.$
\end{corollary}
\begin{proof}
To prove (\ref{eq15}), it is sufficient to consider both sides of the identity from Theorem~\ref{t5}
as meromorphic functions  of variable $a$
and compare corresponding residues at the simple pole $a=\beta+n,$ where $n$ is a non-negative integer, on
both sides. Imposing the restrictions on $\alpha, \beta$ formulated in the hypothesis implies that the residue
on the left-hand side of the identity of Theorem \ref{t5} is equal to $-1,$ and clearly,
we have
\begin{equation*}
\begin{split}
1&=\lim_{a\to n+\beta}(n+\beta-a)\frac{a(\alpha-\beta)((1+\alpha)(1+\beta)+\alpha\beta-2a^2)}{(\alpha^2-a^2)
(\beta^2-a^2)} \\
&\times\sum_{k=n}^{\infty}
\frac{k!(1\pm(\alpha-\beta))_k(1\pm 2a)_k(7/5+3(\alpha+\beta)/10\pm\sqrt{D}/10)_k}%
{(-4)^k(3/2)_k(1+\alpha\pm a)_k(1+\beta\pm a)_k(2/5+3(\alpha+\beta)/10\pm\sqrt{D}/10)_k} \\[3pt]
&=\frac{(n+\beta)(\alpha-\beta)((1+\alpha)(1+\beta)+\alpha\beta-2(n+\beta)^2)}{(\alpha^2-(n+\beta)^2)(n+2\beta)(-1)^n} \\
&\times\sum_{k=n}^{\infty}\binom{k}{n}
\frac{(1\pm(\alpha-\beta))_k(1\pm (2n+2\beta))_k((14+3(\alpha+\beta)\pm\sqrt{D_1})/10)_k}%
{(-4)^k(3/2)_k(1+\alpha\pm (\beta+n))_k(1+2\beta+n)_k((4+3(\alpha+\beta)\pm\sqrt{D_1})/10)_k} \\[3pt]
&=\frac{(n+\beta)(\alpha-\beta)(1\pm(\alpha-\beta))_n
(1\pm(2n+2\beta))_n(10n+4+3(\alpha+\beta)\pm\sqrt{D_1})}%
{20(\alpha^2-(n+\beta)^2)(n+2\beta)4^n(3/2)_n(1+\alpha\pm(\beta+n))_n(1+2\beta+n)_n} \\[5pt]
&\times {}\sb 7F\sb 6\left(\left.\atop{n+1, n+1\pm(\alpha-\beta),  n+1\pm(2n+2\beta),
n+\frac{7}{5}+\frac{3(\alpha+\beta)}{10}
\pm \frac{\sqrt{D_1}}{10}
}{\frac{3}{2}+n, n+\alpha+1\pm (\beta+n), 2n+2\beta+1, n+\frac{2}{5} +\frac{3(\alpha+\beta)}{10}\pm\frac{\sqrt{D_1}}{10}}\right|-\frac{1}{4}\right).
\end{split}
\end{equation*}
Now taking into account that
$$
\frac{\alpha-\beta}{\alpha-\beta-n}\cdot \frac{(1+\beta-\alpha)_n}{(1+\alpha-\beta-n)_n}=(-1)^n,
\qquad \frac{n+\beta}{n+2\beta}\cdot\frac{(1-2n-2\beta)_n}{(1+2\beta+n)_n}=\frac{(-1)^n}{2},
$$
$$
\frac{(10n+4+3(\alpha+\beta))^2-D_1}{20}=(\alpha-\beta+n+1)(3n+2\beta+1),
$$
and $4^n(3/2)_n=(2n+1)!/n!,$ we get the required identity.
\end{proof}
Similarly, from Theorem \ref{t6} we get
\begin{corollary} \label{c12}
Let $n$ be a non-negative integer and  $a, \alpha$ be complex numbers such
that $2\alpha\ne -n, -n-1,\ldots,$ and $\alpha\pm a\ne -n-1, -n-2,\ldots, -2n-1.$
Then we have
\begin{equation}
\begin{split}
{}\sb 7F\sb 6\left(\left.\atop{n+1, 2n+\alpha+1\pm a, 1-\alpha\pm a,
n+\frac{7+3\alpha}{5}
\pm \sqrt{\frac{a^2+(n+\alpha)^2}{5}-(\frac{1-\alpha}{5})^2}
}{\frac{3}{2}+n, 2n+2\alpha+1, n+\alpha+1\pm a, n+\frac{2+3\alpha}{5}
\pm\sqrt{\frac{a^2+(n+\alpha)^2}{5}-(\frac{1-\alpha}{5})^2}}\right|-\frac{1}{4}\right) \\[5pt]
=\frac{(n+2)_{n+1}(n+2\alpha)_{n+1}}{(n+\alpha+1\pm a)_{n+1}}.
\label{eq16}
\end{split}
\end{equation}
\end{corollary}
\begin{proof}
To prove (\ref{eq16}), we rewrite the identity of Theorem \ref{t6} in the form
\begin{equation}
\begin{split}
&\sum_{k=0}^{\infty}\frac{k+\alpha}{((k+\alpha)^2-a^2)((k+\alpha)^2-b^2)}=
\frac{((1+\alpha)^2+\alpha^2-a^2-b^2)}{4(\alpha^2-a^2)(\alpha^2-b^2)} \\
&\times
{}\sb 8F\sb 7\left(\left.\atop{1,  1,  1\pm a\pm b, \frac{7+3\alpha}{5}\pm \sqrt{D(b^2)}
}{\frac{3}{2}, 1+\alpha\pm a, 1+\alpha\pm b, \frac{2+3\alpha}{5}\pm\sqrt{D(b^2)}}\right|-\frac{1}{4}\right),
\label{eq18}
\end{split}
\end{equation}
where
$$
D(b^2):=\frac{a^2+b^2}{5}-\left(\frac{1-\alpha}{5}\right)^2,
$$
replace $b^2$ by $z$ and consider both sides of (\ref{eq18}) as meromorphic functions of variable $z.$
Suppose that $2\alpha$ is distinct from $-n, -n-1, \ldots.$ This restriction ensures that
$(n+\alpha)^2\ne (j+\alpha)^2$ for any non-negative integer  $j\ne n.$
Now
equating residues on both sides of
(\ref{eq18}) at the simple pole $z=(n+\alpha)^2,$ where $n$ is a non-negative integer, we have
\begin{equation*}
\begin{split}
&\frac{n+\alpha}{(n+\alpha)^2-a^2}=\lim_{z\to (n+\alpha)^2}((n+\alpha)^2-z)\cdot \frac{(1+\alpha)^2+\alpha^2-a^2-z}%
{4(\alpha^2-a^2)(\alpha^2-z)} \\[3pt]
&\times
\sum_{k=n}^{\infty}\frac{ k! \bigl(\frac{7+3\alpha}{5}\pm\sqrt{D(z)}\bigr)_k
}{(-4)^k (3/2)_k\bigl(\frac{2+3\alpha}{5}\pm\sqrt{D(z)}\bigr)_k}
\prod_{j=1}^k\frac{((j\pm a)^2-z)}{((j+\alpha)^2-z)((j+\alpha)^2-a^2)} \\[5pt]
&=
\frac{(1+\alpha)^2-a^2-n^2-2n\alpha}{(-1)^{n-1} (2\alpha+n)_n}
\sum_{k=n}^{\infty}\binom{k}{n}\frac{(1\pm a\pm (n+\alpha))_k}{(-4)^{k+1}(3/2)_k(\alpha\pm a)_{k+1}
(2\alpha+2n+1)_{k-n}} \\[5pt]
&\times\frac{\bigl(\frac{7+3\alpha}{5}\pm\sqrt{D((n+\alpha)^2)}\bigr)_k}%
{\bigl(\frac{2+3\alpha}{5}\pm\sqrt{D((n+\alpha)^2)}\bigr)_k}
=\frac{(1\pm a\pm (n+\alpha))_n ((2n+\alpha+1)^2-a^2)}%
{4^{n+1} (3/2)_n(2\alpha+n)_n(\alpha\pm a)_{n+1}} \\[5pt]
&\times
{}\sb 7F\sb 6\left(\left.\atop{n+1, 2n+\alpha+1\pm a, 1-\alpha\pm a,
n+\frac{7+3\alpha}{5}
\pm \sqrt{D((n+\alpha)^2)}
}{\frac{3}{2}+n, 2n+2\alpha+1, n+\alpha+1\pm a, n+\frac{2+3\alpha}{5}
\pm\sqrt{D((n+\alpha)^2)}}\right|-\frac{1}{4}\right).
\end{split}
\end{equation*}
Here in the last equality we used the fact that
$$
\frac{\bigl(\frac{7+3\alpha}{5}\pm\sqrt{D((n+\alpha)^2)}\bigr)_n}%
{\bigl(\frac{2+3\alpha}{5}\pm\sqrt{D((n+\alpha)^2)}\bigr)_n}=
\frac{(2n+\alpha+1)^2-a^2}{(1+\alpha)^2-a^2-n^2-2n\alpha}.
$$
Finally, after simplifying based on the cancelations
$$
\frac{(1-a-n-\alpha)_n}{(\alpha+a)_{n+1}}=\frac{(-1)^n}{\alpha+a+n}, \qquad
\frac{(1+a-n-\alpha)_n}{(\alpha-a)_{n+1}}=\frac{(-1)^n}{\alpha-a+n},
$$
we get the required identity.
\end{proof}
Setting $\alpha=1$ and replacing $n$ by $n-1$ in (\ref{eq16}) we get
\begin{corollary} \label{c13}
Let $n$ be a positive integer and $a$ be a complex number such that $\pm a\ne -n-1, -n-2,\ldots, -2n.$
Then we have
\begin{equation}
{}\sb 7F\sb 6\left(\left.\atop{n, 2n\pm a, \pm a,
n+1
\pm \sqrt{\frac{a^2+n^2}{5}}
}{n+\frac{1}{2}, 2n+1, n+1\pm a, n
\pm\sqrt{\frac{a^2+n^2}{5}}}\right|-\frac{1}{4}\right)=\prod_{j=n+1}^{2n}\frac{j^2}{j^2-a^2}.
\label{eq17}
\end{equation}
\end{corollary}
Formula (\ref{eq17}) generalizes similar identities from \cite[\S 2]{b}. In particular,
substituting $a=\sqrt{c}/n$ gives \cite[Corollary 2]{b} and substituting $a=i\sqrt{b+n^2}$
gives \cite[Corollary 3]{b}.

\vspace{0.5cm}

\section{Supercongruences arising from the Amdeberhan-Zeilberger series for $\zeta(3)$} \label{SS5}

\vspace{0.1cm}

In this section, we consider supercongruences arising from the Amdeberhan-Zeilberger series (\ref{saz})
for $\zeta(3).$ In \cite{gz}, Guillera and Zudilin proposed conjecturally a  $p$-adic analog:
$$
\sum_{k=0}^{p-1}\frac{(\frac{1}{2})_k^5}{(1)_k^5}(205k^2+160k+32)(-1)^k 2^{10k}\equiv 32p^2\pmod{p^5}
\quad\text{for prime}\, p>3.
$$
In \cite{sun2}, Z.-W.~Sun formulated more general conjectures: let $p$ be an odd prime, then
$$
\sum_{k=0}^{p-1}(205k^2+160k+32)(-1)^k\binom{2k}{k}^5\equiv 32p^2 + 64p^3H_{p-1}\pmod{p^7}
\quad\text{for}\, p\ne 5,
$$
where $H_{p-1}=\sum_{k=1}^{p-1}1/k,$ and
\begin{equation}
\sum_{k=0}^{(p-1)/2}(205k^2+160k+32)(-1)^k\binom{2k}{k}^5\equiv 32p^2 + \frac{896}{3}p^5B_{p-3} \pmod{p^6}
\quad\text{for}\, p>3,
\label{eq20}
\end{equation}
where $B_0, B_1, B_2, \ldots$ are Bernoulli numbers.

Moreover, Sun \cite{sun2} introduced the related sequence
$$
a_n=\frac{1}{8n^2\binom{2n}{n}^2}\sum_{k=0}^{n-1}(205k^2+160k+32)(-1)^{n-1-k}\binom{2k}{k}^5
$$
and conjectured  that for any positive integer $n,$ $a_n$ should be a positive integer.

In this section, we confirm these conjectures (with the only exception that we prove (\ref{eq20}) modulo $p^5$)
and prove the following theorems.
\begin{theorem} \label{t7}
Let $n$ be a positive integer and let
$$
A_n:=\sum_{k=0}^{n-1}(-1)^{n-1-k}\binom{2k}{k}^5(205k^2+160k+32).
$$
Then the following two alternative representations are valid:
$$
A_n=16n\binom{2n}{n}\sum_{k=0}^{n-1}\binom{n+k-1}{k}^4(2k+n)
$$
and
$$
A_n=8n^2\binom{2n}{n}^2\sum_{k=0}^{n-1}(-1)^k\binom{2n-1}{n+k}\binom{2n-k-2}{n-k-1}^2.
$$
\end{theorem}
Theorem \ref{t7} implies immediately the following
\begin{corollary} \label{c14}
For any positive integer $n,$ $a_n:=\frac{A_n}{8n^2\binom{2n}{n}^2}$ is a positive integer.
\end{corollary}
\begin{theorem} \label{t8}
Let $p$ be an odd prime. Then the following supercongruences take place:
$$
\sum_{k=0}^{p-1}(205k^2+160k+32)(-1)^k\binom{2k}{k}^5\equiv 32p^2+64p^3H_{p-1}\pmod{p^7}\quad \text{for}\,\, p\ne 5,
$$
$$
\sum_{k=0}^{(p-1)/2}(205k^2+160k+32)(-1)^k\binom{2k}{k}^5\equiv 32p^2\pmod{p^5}\quad \text{for} \,\, p>3.
$$
\end{theorem}
The proof of Theorem \ref{t7} is contained in the following two lemmas.
\begin{lemma} \label{l1}
For any positive integer $N,$ the following identity holds:
$$
A_N=16N\binom{2N}{N}\sum_{k=0}^{N-1}\binom{N+k-1}{k}^4(2k+N).
$$
\end{lemma}
\begin{proof}
Consider the Markov kernel $H_1(n,k)$ defined in (\ref{H1nk}) for $\alpha=1,$ $a=b=0:$
$$
H_1(n,k)=\frac{(k+n)!^4}{(k+2n+1)!^4} (3n+2k+2).
$$
Then the corresponding to it WZ pair has the form
$$
F(n,k)=\frac{(-1)^n n!^6}{(2n)!} H_1(n,k), \qquad
G(n,k)=\frac{(-1)^nn!^6(k+n)!^4}{2(2n+1)!(k+2n+2)!^4} q(n,k),
$$
where
\begin{equation*}
\begin{split}
q(n,k)&=(n+1)^4(205n^2+250n+77)+k(254+344k+1526n+3628n^2
+888k^2n \\
&+2928kn^2+4268n^3+248k^2+101k^3
+1648kn+574n^5+2486n^4
+22k^4+2k^5 \\
&+664kn^4+2288kn^3+408k^2n^3+1048k^2n^2+141k^3n^2+240k^3n+26k^4n).
\end{split}
\end{equation*}
It is easy to show (see Proposition C or \cite[Ch.~7]{pwz}) that the pair $(\bar{F}, \bar{G})$
given by
$$
\bar{G}(n,k)=(-1)^{k-1}k\binom{2k}{k}\binom{n+2k-1}{n+k}^4 (3k+2n),
$$
$$
\bar{F}(n,k)=(-1)^k\binom{2k}{k}\left(\frac{(n+2k-1)!}{k!(n+k)!}\right)^4 q_1(n,k)
$$
with
\begin{equation*}
\begin{split}
q_1&(n,k)=205k^6+160k^5+32k^4+2n^6+n^5(4+26k)+n^4(2+42k+141k^2) \\
&+n^3k(16+176k+408k^2)+n^2k^2(48+368k+664k^2)
+nk^3(64+384k+574k^2)
\end{split}
\end{equation*}
is its dual WZ pair, for which
\begin{equation}
\bar{F}(n+1,k)-\bar{F}(n,k)=\bar{G}(n,k+1)-\bar{G}(n,k).
\label{eq19}
\end{equation}
It turns out for our further proof that it is useful to consider the usual binomial coefficient
$\binom{r}{k}$ in a more general setting, i.e., to allow an arbitrary real number to appear in the
upper index of $\binom{r}{k},$ and to allow an arbitrary integer in the lower.
We give the following formal definition (see \cite[\S 5.1]{cm}:
\begin{equation*}
\binom{r}{k} =\begin{cases}
 \frac{r(r-1)\cdots(r-k+1)}{k(k-1)\cdots 1},     & \quad \text{if integer} \quad k\ge 0; \\
0, & \quad\text{if integer} \quad k<0.
\end{cases}
\end{equation*}
After this elaboration we see that $\bar{G}(n,k)$ is defined for all integers $n, k.$  Rewriting $\bar{F}(n,k)$
in the form
\begin{equation*}
\bar{F}(n,k) =\begin{cases}
 \frac{(-1)^k}{k^4}\binom{2k}{k}\binom{n+2k-1}{n+k}^4 q_1(n,k), & \quad \text{if} \quad k\ne 0; \\[3pt]
 2(n+1)^2, & \quad \text{if} \quad k=0, n\ge 0; \\
0, & \quad\text{if} \quad k=0, n<0;
\end{cases}
\end{equation*}
we can conclude that $\bar{F}(n,k)$ is well defined for all $n,k\in {\mathbb Z}.$
Now we can show that relation (\ref{eq19}) takes place for all integers $n$ and $k.$
Indeed, if $k<0$ or if $k=0$ and $n<-1,$ then all parts in (\ref{eq19}) are zero.
If $k=0$ and $n\ge -1,$ then $\bar{G}(n,k)=0,$ $\bar{G}(n,k+1)=2(2n+3),$ $\bar{F}(n+1,k)=2(n+2)^2,$
$F(n,k)=(n+1)^2,$
and relation (\ref{eq19}) is equivalent to the obvious equality
$$
2(n+2)^2-(n+1)^2=2(2n+3).
$$
If $k>0$ and $n+k<0,$ then $\bar{F}(n,k)=\bar{G}(n,k)=0,$
$$
\bar{F}(n+1,k)=\frac{(-1)^k}{k^4}\binom{2k}{k}\binom{n+2k}{n+k+1}^4 q_1(n+1,k),
$$
and
\begin{equation*}
\bar{G}(n,k+1)=(-1)^k(k+1)\binom{2k+2}{k+1}\binom{n+2k+1}{n+k+1}^4(3k+2n+3).
\end{equation*}
If moreover, $n+k+1<0,$ then $\bar{F}(n+1,k)=\bar{G}(n,k+1)=0$ and (\ref{eq19}) holds.
If $n+k+1=0,$ then $\bar{G}(n,k+1)=(-1)^k(k+1)^2\binom{2k+2}{k+1}$ and
$$
\bar{F}(n+1,k)=\frac{(-1)^k}{k^4}\binom{2k}{k} q_1(-k,k)=(-1)^k\binom{2k}{k}(2k+2)(2k+1)=\bar{G}(n,k+1),
$$
and therefore (\ref{eq19}) holds. If $k>0$ and $n+k\ge 0,$ then canceling common factorials
on both sides of (\ref{eq19}) we get the equality
\begin{equation*}
\begin{split}
&(n+2k)^4 q_1(n+1,k)-(n+k+1)^4 q_1(n,k) \\
&=2(2k+1)(n+2k)^4(n+2k+1)^4(3k+2n+3)+k^5(n+k+1)^4(3k+2n),
\end{split}
\end{equation*}
which can be easily checked by straightforward verification.

Now let $N\in {\mathbb N}.$ Considering relation (\ref{eq19}) at the point $(n-N,k)$ we have
\begin{equation}
\bar{F}(n+1-N,k)-\bar{F}(n-N,k)=\bar{G}(n-N,k+1)-\bar{G}(n-N,k).
\label{eq21}
\end{equation}
Summing both sides of (\ref{eq21}) over $k$ from $0$ to $N-1$ we have
\begin{equation}
\sum_{k=0}^{N-1}(\bar{F}(n+1-N,k)-\bar{F}(n-N,k))=\bar{G}(n-N,N)-\bar{G}(n-N,0)=\bar{G}(n-N,N).
\label{eq22}
\end{equation}
Now summing (\ref{eq22}) over $n$ from $0$ to $N-1$ we get
$$
\sum_{k=0}^{N-1}(\bar{F}(0,k)-\bar{F}(-N,k))=\sum_{n=0}^{N-1}\bar{G}(n-N,N).
$$
Since $\bar{F}(-N,k)=0$ for $k=0,1,\ldots, N-1,$ we obtain
$$
\sum_{k=0}^{N-1}\bar{F}(0,k)=\sum_{n=0}^{N-1}\bar{G}(n-N,N)
$$
or
$$
\frac{1}{16}\sum_{k=0}^{N-1}(-1)^k\binom{2k}{k}^5(205k^2+160k+32)=(-1)^{N-1}
N\binom{2N}{N}\sum_{n=0}^{N-1}\binom{N+n-1}{n}^4(2n+N),
$$
and the lemma is proved.
\end{proof}
\begin{lemma} \label{l2}
For any positive integer $N,$ the following identity holds:
$$
A_N=8N^2\binom{2N}{N}^2\cdot\sum_{k=0}^{N-1}(-1)^k\binom{2N-1}{N+k}\binom{2N-k-2}{N-1-k}^2.
$$
\end{lemma}
\begin{proof}
Let $N\in {\mathbb Z},$ $N\ge 0.$ Put
\begin{equation}
S_N:=2\sum_{k=0}^N\binom{k+N}{k}^4(N+2k+1).
\label{eq23}
\end{equation}
Now rewriting $S_N$ in the form of a terminating hypergeometric series, we get
$$
S_N=2(N+1)\sum_{k=0}^N\frac{(N+1)_k^4\, (\frac{N+3}{2})_k}{(1)_k^4\, (\frac{N+1}{2})_k}.
$$
Changing the order of summation and noticing that
$$
(\alpha)_{N-k}=\frac{(-1)^k (\alpha)_N}{(1-\alpha-N)_k}
$$
we get
\begin{equation}
S_N=2\binom{2N}{N}^4(3N+1)\cdot
{}\sb 6F\sb 5\left(\left.\atop{1, \, \frac{1}{2}-\frac{3}{2}N, -N, -N, -N, -N}
{-\frac{1}{2}-\frac{3}{2}N, -2N, -2N, -2N, -2N}\right|1\right).
\label{eq24}
\end{equation}
To evaluate the hypergeometric series on the right-hand side of (\ref{eq24}), we apply
Whipple's transformation \cite[(7.7)]{w} which transforms a Saalschutian ${}\sb 4F\sb 3(1)$
series into a well-poised ${}\sb 7F\sb 6(1)$ series (see \cite[p.~61, (2.4.1.1)]{sl}):
\begin{equation}
\begin{split}
{}\sb 4F\sb 3&\left(\left.\atop{f-a_1-a_2,  d_1, d_2, -N}
{f-a_1, f-a_2, g}\right|1\right)=
\frac{(g-d_1)_N(g-d_2)_N}{(g)_N(g-d_1-d_2)_N} \\[5pt]
&\times {}\sb 7F\sb 6\left(\left.\atop{f-1,  \frac{1}{2}f+\frac{1}{2}, a_1, a_2, d_1, d_2, -N}
{\frac{1}{2}f-\frac{1}{2}, f-a_1, f-a_2, f-d_1, f-d_2, f+N}\right|1\right).
\label{eq25}
\end{split}
\end{equation}
Setting $a_1=a_2=-N,$ $d_1=-N,$ $d_2=1,$ $f=-3N$ in (\ref{eq25}), we get
\begin{equation}
\begin{split}
{}\sb 6F\sb 5&\left(\left.\atop{1, \, \frac{1}{2}-\frac{3}{2}N, -N, -N, -N, -N}
{-\frac{1}{2}-\frac{3}{2}N, -2N, -2N, -2N, -2N}\right|1\right) \\[5pt]
&=\frac{(2N+1)^2}{(3N+1)(N+1)}
{}\sb 4F\sb 3\left(\left.\atop{1, \,  -N, -N, -N}
{-2N, -2N, N+2}\right|1\right).
\label{eq26}
\end{split}
\end{equation}
Therefore from (\ref{eq24}) and (\ref{eq26}) we get
$$
S_N=2\binom{2N}{N}^4\frac{(2N+1)^2}{N+1}\cdot
{}\sb 4F\sb 3\left(\left.\atop{1, \,  -N, -N, -N}
{-2N, -2N, N+2}\right|1\right)
$$
or
\begin{equation}
S_N=\frac{2(2N+1)!^2 (2N)!^2}{(N+1)! N!^7}\sum_{k=0}^N\frac{(-N)_k^3}{(-2N)_k^2(N+2)_k}.
\label{eq27}
\end{equation}
Replacing Pocchammer's symbols by factorials in (\ref{eq27}) we arrive at
\begin{equation}
S_N=(N+1)\binom{2N+2}{N+1}\sum_{k=0}^N(-1)^k\binom{2N+1}{N-k}\binom{2N-k}{N-k}^2.
\label{eq28}
\end{equation}
Now replacing $N$ by $N-1$ in (\ref{eq28}), and using (\ref{eq23}) and Lemma \ref{l1}
we get the required identity.
\end{proof}
To prove Theorem \ref{t8}, we need several results concerning harmonic sums modulo a power of prime $p.$
The multiple harmonic sum is defined by
$$
H(a_1, a_2, \ldots, a_r;n)=\sum_{1\le k_1<k_2<\dots<k_r\le n} \frac{1}{k_1^{a_1}k_2^{a_2}\cdots k_r^{a_r}},
$$
where $n\ge r\ge 1$ and $(a_1,a_2,\dots,a_r)\in {\mathbb N}^r.$
For $r=1$ we will also use the notation
$$
H_n^{(a)}:=H(a;n)=\sum_{k=1}^n\frac{1}{k^a} \quad\text{and}\quad
H_n:=H_n^{(1)}.
$$
The values of many harmonic sums modulo a power of prime $p$ are well known.
We need the following results.
\begin{lemma} \label{l3} \cite[Theorem 5.1]{s0}
Let $p$ be a prime greater than $5.$ Then
$$
H_{p-1}\equiv H_{p-1}^{(3)}\equiv 0\pmod{p^2}, \qquad
H_{p-1}^{(2)}\equiv H_{p-1}^{(4)}\equiv 0\pmod p.
$$
\end{lemma}
\begin{lemma} \label{l4}
Let $p>5$ be a prime. Then
$$
H(\{1\}^2;p-1)\equiv -\frac{1}{2}H_{p-1}^{(2)}\pmod{p^4}, \qquad
H(\{1\}^3;p-1)\equiv 0\pmod{p^2},
$$
$$
H(\{1\}^4;p-1)\equiv 0\pmod p.
$$
\end{lemma}
\begin{proof}
Since for $n\ge 1$ (see, for example, \cite{t2})
$$
H(\{1\}^2;n)=\frac{1}{2}(H_n^2-H_n^{(2)}),
$$
$$
H(\{1\}^3;n)=\frac{1}{6}(H_n^3-3H_nH_n^{(2)}+2H_n^{(3)}),
$$
$$
H(\{1\}^4;n)=\frac{1}{24}(H_n^4-6H_n^2H_n^{(2)}+8H_nH_n^{(3)}+3(H_n^{(2)})^2-6H_n^{(4)}),
$$
then by Lemma \ref{l3}, we get the required congruences.
\end{proof}
\begin{lemma} \label{l5}
Let $p>5$ be a prime. Then
$$
\frac{1}{2}\binom{2p}{p}\equiv 1+pH_{p-1}-\frac{p^2}{2}H_{p-1}^{(2)}\pmod{p^5}.
$$
\end{lemma}
\begin{proof}
It is readily seen that
$$
\frac{1}{2}\binom{2p}{p}=\binom{2p-1}{p-1}=\prod_{j=1}^{p-1}\left(1+\frac{p}{j}\right)
=\sum_{j=0}^{p-1}p^j H(\{1\}^j;p-1)
$$
and therefore
$$
\frac{1}{2}\binom{2p}{p}\equiv 1+pH_{p-1}+p^2H(\{1\}^2;p-1)+p^3H(\{1\}^3;p-1)+p^4H(\{1\}^4;p-1)\pmod{p^5}.
$$
Now by Lemma \ref{l4}, we get the required congruence.
\end{proof}
For similar congruences related to the central binomial coefficients, see \cite{t1}.

\subsection*{Proof of Theorem \ref{t8}.}

From Theorem \ref{t7} with $n=p,$ where $p>5$ is a prime, we have
\begin{equation}
\sum_{k=0}^{p-1}(-1)^k\binom{2k}{k}^5(205k^2+160k+32)=8p^2\binom{2p}{p}^2\sum_{k=0}^{p-1}(-1)^k
\binom{2p-1}{p+k}\binom{2p-k-2}{p-1-k}^2.
\label{eq29}
\end{equation}
For the sum on the right-hand side of (\ref{eq29}), changing the order of summation we have
\begin{equation}
\begin{split}
&\sum_{k=0}^{p-1}(-1)^k
\binom{2p-1}{p+k}\binom{2p-k-2}{p-1-k}^2=\sum_{k=0}^{p-1}(-1)^k\binom{2p-1}{k}\binom{p-1+k}{k}^2 \\
&=1+\sum_{k=1}^{p-1}(-1)^k\frac{(2p-1)(2p-2)\cdots(2p-k)}{k!}\frac{p^2(p+1)^2\cdots(p+k-1)^2}{k!^2} \\
&=1+p^2\sum_{k=1}^{p-1}\frac{1}{k^2}\prod_{j=1}^k\left(1-\frac{2p}{j}\right)
\prod_{j=1}^{k-1}\left(1+\frac{p}{j}\right)^2.
\label{eq30}
\end{split}
\end{equation}
Since
\begin{equation*}
\prod_{j=1}^k\left(1-\frac{2p}{j}\right)\equiv 1-2pH_k+4p^2H(\{1\}^2;k)=1-2pH_k+2p^2(H_k^2-H_k^{(2)})\pmod{p^3}
\end{equation*}
and
\begin{equation*}
\begin{split}
\prod_{j=1}^{k-1}\left(1+\frac{p}{j}\right)^2&=\prod_{j=1}^{k-1}\left(1+\frac{2p}{j}+\frac{p^2}{j^2}\right)
\equiv 1+2pH_{k-1}+p^2H_{k-1}^{(2)}+4p^2H(\{1\}^2;k-1) \\
&=1+2pH_{k-1}+2p^2H_{k-1}^2-p^2H_{k-1}^{(2)}\pmod{p^3},
\end{split}
\end{equation*}
substituting these congruences in (\ref{eq30}) and simplifying we obtain
\begin{equation}
\sum_{k=0}^{p-1}(-1)^k\binom{2p-1}{p+k}\binom{2p-k-2}{p-1-k}^2\equiv 1+p^2H_{p-1}^{(2)}-2p^3H_{p-1}^{(3)}
-3p^4\sum_{k=1}^{p-1}\frac{H_{k-1}^{(2)}}{k^2}\pmod{p^5}.
\label{eq31}
\end{equation}
Note that
\begin{equation*}
\begin{split}
\sum_{k=1}^{p-1}\frac{H_{k-1}^{(2)}}{k^2}&=\sum_{k=1}^{p-1}\frac{1}{k^2}\sum_{j=1}^{k-1}\frac{1}{j^2}
=\sum_{k=1}^{p-1}\frac{1}{k^2}\sum_{j=1}^k\frac{1}{j^2}-H_{p-1}^{(4)} \\
&=
\sum_{j=1}^{p-1}\frac{1}{j^2}\Bigl(H_{p-1}^{(2)}-H_{j-1}^{(2)}\Bigr)-H_{p-1}^{(4)}=
\Bigl(H_{p-1}^{(2)}\Bigr)^2-\sum_{j=1}^{p-1}\frac{H_{j-1}^{(2)}}{j^2}-H_{p-1}^{(4)}
\end{split}
\end{equation*}
and therefore,
\begin{equation}
2\sum_{k=1}^{p-1}\frac{H_{k-1}^{(2)}}{k^2}=\Bigl(H_{p-1}^{(2)}\Bigr)^2-H_{p-1}^{(4)}.
\label{eq32}
\end{equation}
Now by Lemma \ref{l3}, from (\ref{eq31}) and (\ref{eq32}) we have
\begin{equation}
\sum_{k=0}^{p-1}(-1)^k\binom{2p-1}{p+k}\binom{2p-k-2}{p-1-k}^2\equiv 1+p^2H_{p-1}^{(2)}\pmod{p^5}.
\label{eq33}
\end{equation}
From Lemma \ref{l5} we find easily
\begin{equation}
\binom{2p}{p}^2\equiv 4\Bigl(1+2pH_{p-1}-p^2H_{p-1}^{(2)}\Bigr)\pmod{p^5}.
\label{eq34}
\end{equation}
Now from (\ref{eq29}), (\ref{eq33}) and (\ref{eq34}), by Lemma \ref{l3},  for any prime $p>5,$ we have
\begin{equation*}
\begin{split}
\sum_{k=0}^{p-1}(-1)^k\binom{2k}{k}^5(205k^2+160k+32)&\equiv
32p^2\Bigl(1+p^2H_{p-1}^{(2)}\Bigr)\Bigl(1+2pH_{p-1}-p^2H_{p-1}^{(2)}\Bigr) \\
&\equiv 32p^2+64p^3H_{p-1}\pmod{p^7}.
\end{split}
\end{equation*}
The validity of this congruence for $p=3$ can be easily checked by straightforward verification.
Taking into account that for an odd prime $p,$
$$
\binom{2k}{k}^5\equiv 0\pmod{p^5}
$$
for  $k=\frac{p+1}{2},\dots,p-1,$ and applying Lemma \ref{l3}, we get the second  congruence of Theorem~\ref{t8}. \qed

\section{A supercongruence arising from a series for the constant $K.$}
\label{SS6}

\vspace{0.1cm}

In  Section \ref{SS3} we proved the accelerated formula
$$
K=\sum_{n=1}^{\infty}\frac{(-27)^{n-1}(15n-4)}{n^3\binom{2n}{n}^2\binom{3n}{n}},
$$
which was earlier  conjectured  by Z.-W.~Sun. Motivated by this series, Z.-W.~Sun \cite[Conj.~5.6]{sun1}
formulated the following conjecture on supercongruences: for any prime $p>3$ and a positive integer $a,$
$$
\sum_{k=0}^{p^a-1}\frac{15k+4}{(-27)^k}\binom{2k}{k}^2\binom{3k}{k}\equiv 4\left(\frac{p^a}{3}\right) p^a\pmod{p^{2+a}}.
$$
In this connection we prove here the following theorem.
\begin{theorem} \label{t9}
Let $p$ be a prime greater than $3.$ Then
$$
\sum_{k=0}^{p-1}\frac{15k+4}{(-27)^k}\binom{2k}{k}^2\binom{3k}{k}\equiv 4\left(\frac{p}{3}\right) p\pmod{p^{2}}.
$$
\end{theorem}
\begin{proof}
Consider the WZ pair $(F,G)$ defined in the proof of Theorem \ref{t1} with $\alpha=1/3,$ $\beta=2/3,$ $a=0:$
$$
F(n,k)=\frac{(-1)^n(3n+1)! n!^3}{(2n)! 27^n}\cdot\frac{\left(\frac{1}{2}\right)_k^2\left(\frac{2}{3}\right)_k^2(n+2k+1)}%
{\left(\frac{1}{3}\right)^2_{n+k+1}\left(\frac{2}{3}\right)^2_{n+k+1}},
$$
$$
G(n,k)=F(n,k)\cdot\frac{45n^2+63n+22+18k(3n+k+2)}{18(2n+1)(n+1+2k)}.
$$
Then it is readily seen that the pair $(\bar{F},\bar{G})$ given by
$$
\bar{F}(n,k)=\frac{(-27)^{-k}(2k)!(3k+3n)!^2n!^2}{(3k+1)!k!^3(k+n)!^2(3n)!^2}\,((15k+4)(3k+1)+18n(3k+n+1)),
$$
$$
\bar{G}(n,k)=\frac{3k^3(3k-1)(-27)^{1-k}(2k)!(3k+3n)!^2n!^2(2n+k+1)}{(3k)!k!^3(k+n)!^2(3n+2)!^2}
$$
is its dual WZ pair (see Proposition C or \cite[Ch.~7]{pwz}), for which we have
\begin{equation}
\bar{F}(n+1,k)-\bar{F}(n,k)=\bar{G}(n,k+1)-\bar{G}(n,k), \qquad n,k\ge 0.
\label{eq35}
\end{equation}
Summing (\ref{eq35}) over $k=0,1,\dots,p-1$ and observing that $\bar{G}(n,0)=0,$ we obtain
\begin{equation}
\sum_{k=0}^{p-1}\bar{F}(n+1,k)-\sum_{k=0}^{p-1}\bar{F}(n,k)=\bar{G}(n,p).
\label{eq36}
\end{equation}
Further, for every integer $n$ satisfying $0\le n<\frac{p-2}{3}$ we have
\begin{equation}
\bar{G}(n,p)=\frac{(-27)^{1-p}(2p)! (3p+3n)!^2 n!^2 (2n+p+1)}{(3p-2)!(p-1)!^2p!(p+n)!^2(3n+2)!^2}
\equiv 0\pmod{p^3},
\label{eq37}
\end{equation}
since the numerator is divisible by $p^8$ and the denominator is divisible by $p^5$
and not divisible by $p^6.$ Now from (\ref{eq36}) and (\ref{eq37}) for any non-negative integer
$n<\frac{p-2}{3},$ we have
\begin{equation}
\sum_{k=0}^{p-1}\bar{F}(0,k)\equiv \sum_{k=0}^{p-1}\bar{F}(1,k)\equiv\dots\equiv
\sum_{k=0}^{p-1}\bar{F}(n+1,k) \pmod{p^3}.
\label{eq38}
\end{equation}
Moreover, from (\ref{eq38}) we obtain
\begin{equation}
\bar{F}(0,k) \equiv\begin{cases}
\sum_{k=0}^{p-1}\bar{F}\left(\frac{p-1}{3},k\right)\pmod{p^3}, & \quad \text{if} \quad p\equiv 1\pmod{3}; \\[3pt]
\sum_{k=0}^{p-1}\bar{F}\left(\frac{p-2}{3},k\right)\pmod{p^3}, & \quad \text{if} \quad p\equiv 2\pmod{3}.
\label{eq39}
\end{cases}
\end{equation}
Since
$$
\sum_{k=0}^{p-1}\bar{F}(0,k)=\sum_{k=0}^{p-1}\frac{(15k+4)}{(-27)^k}\binom{2k}{k}^2\binom{3k}{k},
$$
it is sufficient to prove the required congruence for the right-hand side of (\ref{eq39}).
We consider separately two cases depending on the sign of $(\frac{p}{3}).$

First, let $p\equiv 2\pmod3,$ then we have
$$
\bar{F}\left(\frac{p-2}{3},k\right)=
\frac{(-27)^{-k}(2k)!(3k+p-2)!^2\left(\frac{p-2}{3}\right)!^2}%
{(3k+1)!k!^3\left(k+\frac{p-2}{3}\right)!^2(p-2)!^2}
((15k+4)(3k+1)+2(p-2)(9k+p+1)).
$$
Note that if $1\le k\le \frac{p-2}{3},$ then the denominator of $\bar{F}(\frac{p-2}{3},k)$ is not divisible by $p$
and the numerator is divisible by $p^2.$ Therefore, we have
\begin{equation}
\bar{F}\left(\frac{p-2}{3},k\right)\equiv 0\pmod{p^2}, \qquad k=1,2,\dots,\frac{p-2}{3}.
\label{eq40}
\end{equation}
Let ${\rm ord}_p\, n$ be the $p$-adic order of $n$ that is the exponent of the highest power of $p$ dividing $n.$
It is clear that
\begin{equation}
{\rm ord}_p\,\bar{F}\left(\frac{p-2}{3},\frac{p+1}{3}\right)={\rm ord}_p\,
\frac{3^{-p-1}\left(\frac{2p+2}{3}\right)!(2p-1)!^2\left(\frac{p-2}{3}\right)!^2}{(p+2)!\left(
\frac{p+1}{3}\right)!^3\left(\frac{2p-1}{3}\right)!^2(p-2)!^2}=1.
\label{eq41}
\end{equation}
Similarly, considering the disjointed intervals $\frac{p+4}{3}\le k\le\frac{p-1}{2},$ $\frac{p+1}{2}\le k\le \frac{2p-1}{3},$
and $\frac{2p+2}{3}\le k\le p-1,$ we obtain
\begin{equation}
{\rm ord}_p\,\bar{F}\left(\frac{p-2}{3},k\right)\ge 3, \qquad k=\frac{p+4}{3}, \frac{p+4}{3}+1, \dots, p-1.
\label{eq42}
\end{equation}
Thus from (\ref{eq40})--(\ref{eq42}) we have
$$
\sum_{k=0}^{p-1}\bar{F}\left(\frac{p-2}{3},k\right)\equiv \bar{F}\left(\frac{p-2}{3},0\right)+
\bar{F}\left(\frac{p-2}{3},\frac{p+1}{3}\right)
\pmod{p^2}
$$
or
\begin{equation*}
\begin{split}
\sum_{k=0}^{p-1}\bar{F}\left(\frac{p-2}{3},k\right)&\equiv 4+2(p-2)(p+1)+\frac{3^{-p-1}\left(\frac{2p+2}{3}\right)!
(2p-1)!^2\left(\frac{p-2}{3}\right)!^2}{(p+2)!\left(\frac{p+1}{3}\right)!^3\left(\frac{2p-1}{3}\right)!^2(p-2)!^2} \\
&\times((5p+9)(p+2)+8(p-2)(p+1))\pmod{p^2}.
\end{split}
\end{equation*}
Taking into account that
\begin{equation*}
\begin{split}
\frac{(2p-1)!^2}{(p-2)!^2(p+2)!}&=\frac{(p-1)^2p}{(p+1)(p+2)}\frac{(p+1)^2(p+2)^2\cdots(2p-1)^2}%
{(p-1)!} \\
&\equiv \frac{(p-1)^2p}{(p+1)(p+2)}\cdot (p-1)!\pmod{p^2}\equiv\frac{p!}{2}\pmod{p^2},
\end{split}
\end{equation*}
$$
(5p+9)(p+2)+8(p-2)(p+1)\equiv 2\pmod p,
$$
and simplifying, we get
$$
\sum_{k=0}^{p-1}\bar{F}\left(\frac{p-2}{3},k\right)\equiv -2p+\frac{2\cdot 3^{-p}\cdot p!}{\left(\frac{2p-1}{3}\right)!
\left(\frac{p+1}{3}\right)!}\pmod{p^2}.
$$
For primes $p>3,$ by Fermat's theorem, we have $3^p\equiv 3\pmod{p}$ and therefore,
\begin{equation}
\sum_{k=0}^{p-1}\bar{F}\left(\frac{p-2}{3},k\right)\equiv
-2p+\frac{2}{3}\frac{p!}{\left(\frac{2p-1}{3}\right)!\left(\frac{p+1}{3}\right)!}\pmod{p^2}.
\label{eq43}
\end{equation}
Now put $n=\frac{p+1}{3}$ and note that
\begin{equation}
\begin{split}
\frac{p\,!}{\left(\frac{2p-1}{3}\right)!\left(\frac{p+1}{3}\right)!}&=\binom{p}{n}
=\frac{p(p-1)\cdots(p-n+1)}{n!}=p\frac{(-1)^{n-1}}{n}\prod_{j=1}^{n-1}\left(1-\frac{p}{j}\right) \\
&\equiv\frac{(-1)^{n-1}}{n} p\pmod{p^2}\equiv -3p\pmod{p^2},
\label{eq44}
\end{split}
\end{equation}
where in the last equality we used the fact that $n-1=\frac{p-2}{3}$ is odd (otherwise, a prime $p>3$ must be even).
Finally, substituting (\ref{eq44}) in (\ref{eq43}) we obtain
$$
\sum_{k=0}^{p-1}\bar{F}\left(\frac{p-2}{3},k\right)\equiv -2p-2p=-4p=4\left(\frac{p}{3}\right) p\pmod{p^2},
$$
which proves the theorem for primes $p\equiv 2\pmod{3}.$

Now suppose that $p\equiv 1\pmod{3}.$ Then we have
$$
\sum_{k=0}^{p-1}\bar{F}(0,k)\equiv\sum_{k=0}^{p-1}\bar{F}\left(\frac{p-1}{3},k\right)\pmod{p^3},
$$
where
$$
\bar{F}\left(\frac{p-1}{3},k\right)=\frac{(-27)^{-k}(2k)!(3k+p-1)!^2\left(\frac{p-1}{3}\right)!^2}%
{(3k+1)!k!^3\left(k+\frac{p-1}{3}\right)!^2(p-1)!^2}((15k+4)(3k+1)+2(p-1)(9k+p+2)).
$$
Note that if $1\le k<\frac{p-1}{3},$ then the denominator of $\bar{F}(\frac{p-1}{3},k)$
is not divisible by $p.$ On the other hand, ${\rm ord}_p\,(3k+p-1)!^2=2$  and therefore,
$\bar{F}(\frac{p-1}{3},k)\equiv 0\pmod{p^2}.$ It is clear that
$$
{\rm ord}_p\,\bar{F}\left(\frac{p-1}{3},\frac{p-1}{3}\right)=1.
$$
Similarly, if $\frac{p+2}{3}\le k\le\frac{2p-2}{3},$ then ${\rm ord}_p\,\bar{F}(\frac{p-1}{3},k)\ge 3,$
and if $\frac{2p+1}{3}\le k\le p-1,$ then ${\rm ord}_p\,\bar{F}(\frac{p-1}{3},k)=3.$
Therefore, we have
\begin{equation}
\sum_{k=0}^{p-1}\bar{F}\left(\frac{p-1}{3},k\right)\equiv \bar{F}\left(\frac{p-1}{3},0\right)
+\bar{F}\left(\frac{p-1}{3},\frac{p-1}{3}\right)\pmod{p^2},
\label{eq45}
\end{equation}
where
\begin{equation}
\bar{F}\left(\frac{p-1}{3},0\right)=4+2(p-1)(p+2)\equiv 2p\pmod{p^2}
\label{eq46}
\end{equation}
and
$$
\bar{F}\left(\frac{p-1}{3},\frac{p-1}{3}\right)=\frac{3^{1-p} (2p-2)!^2 (p(5p-1)+2(p-1)(4p-1))}%
{\left(\frac{2p-2}{3}\right)!\left(\frac{p-1}{3}\right)!p!(p-1)!^2}.
$$
Noting that
$$
\frac{(2p-2)!^2}{p!(p-1)!^2}=\frac{p}{p-1}\frac{(p+1)^2(p+2)^2\cdots(2p-2)^2}{(p-2)!}
\equiv\frac{p}{p-1}\,(p-2)!\equiv p!\pmod{p^2},
$$
and applying Fermat's theorem, we have
$$
\bar{F}\left(\frac{p-1}{3},\frac{p-1}{3}\right)\equiv
\frac{2\cdot p!}{\left(\frac{2p-2}{3}\right)!\left(\frac{p-1}{3}\right)!}\pmod{p^2}.
$$
Now setting $n=\frac{p-1}{3}$ and taking into account that
$$
\frac{(p-1)!}{\left(\frac{2p-2}{3}\right)!\left(\frac{p-1}{3}\right)!}=\binom{p-1}{n}
=(-1)^n\prod_{j=1}^n\left(1-\frac{p}{j}\right)
\equiv 1\pmod{p}
$$
we have
\begin{equation}
\bar{F}\left(\frac{p-1}{3},\frac{p-1}{3}\right)\equiv 2p\pmod{p^2}.
\label{eq47}
\end{equation}
Now by (\ref{eq45})--(\ref{eq47}), we obtain
$$
\sum_{k=0}^{p-1}\bar{F}\left(\frac{p-1}{3},k\right)\equiv 4p=4\left(\frac{p}{3}\right)p\pmod{p^2},
$$
which proves the theorem for primes $p\equiv 1\pmod{3}.$
\end{proof}


\begin{thebibliography}{99}


\bibitem{algr} G.~Almkvist, A.~Granville, {\it Borwein and Bradley's
Ap\'ery-like formulae for $\zeta(4n+3),$} Experiment. Math. {\bf 8} (1999), 197--203.

\bibitem{az}
T. Amdeberhan, D. Zeilberger, {\it Hypergeometric series
acceleration via the WZ method,} Electron. J. Combinatorics {\bf
4(2)} (1997), $\#$R3.

\bibitem{bb}
J.~M.~Borwein, D.~M.~Bradley, {\it Empirically
determined  Ap\'ery-like formulae  for $\zeta(4n+3),$}
Experiment. Math. {\bf 6} (1997), no. 3, 181--194.

\bibitem{bb1}
J.~M.~Borwein, D.~M.~Bradley, {\it Searching symbolically for Ap\'ery-like
formulae for values of the Riemann zeta function,} SIGSAM Bulletin of Algebraic and Symbolic
Manipulation, Association of Computing Machinery, {\bf 30} (1996), no. 2, 2--7.

\bibitem{bbb}
D.~H.~Bailey, J.~M.~Borwein, D.~M.~Bradley, {\it Experimental
determination of Ap\'ery-like identities for $\zeta(2n+2),$}
Experiment. Math. {\bf 15} (2006), no. 3, 281--289.

\bibitem{be}
H.~Bateman and A.~Erd\'elyi, {\it Higher Transcendental Functions,}
Vol.~1, McGraw-Hill, New York, 1953.

\bibitem{b}
D.~M.~Bradley, {\it Hypergeometric functions related to series acceleration formulas,}
Contemporary Math. {\bf 457} (2008), 113--125.

\bibitem{ge}
I.~M.~Gessel, {\it Finding identities with the WZ method,} J.~Symbolic Computation
{\bf 20} (1995), 537--566.

\bibitem{cm}
R.~L.~Graham, D.~E.~Knuth, O.~Patashnik, {\it Concrete Mathematics. A Foundation for Computer Science.}
Second Ed., 1998, Addison-Wesley.

\bibitem{gz}
J.~Guillera, W.~Zudilin, {\it ``Divergent'' Ramanujan-type supercongruences,}
arXiv:1004.4337 [math.NT]

\bibitem{he1}
Kh.~Hessami Pilehrood, T.~Hessami Pilehrood, {\it Generating
function identities for $\zeta(2n+2), \zeta(2n+3)$ via the
WZ-method,} Electron. J. Combin. 15 (2008), no. 1, Research Paper
35, 9 pp.

\bibitem{he2}
Kh.~Hessami Pilehrood, T.~Hessami Pilehrood, {\it Simultaneous
generation for zeta values by the Markov-WZ method,} Discrete Math.
Theor. Comput. Sci. 10 (2008), no. 3, 115--123.

\bibitem{he2.5}
Kh.~Hessami Pilehrood, T.~Hessami Pilehrood, {\it Series acceleration formulas for beta values,}
Discrete Math. Theor. Comput. Sci. 12 (2010), no. 2, 223--236.

\bibitem{he3}
Kh.~Hessami Pilehrood, T.~Hessami Pilehrood, {\it A $q$-analogue of the
Bailey-Borwein-Bradley identity,} J.~Symbolic Computation {\bf 46} (2011),
699--711.

\bibitem{ko}
M.~Koecher, {\it Letter (German),} Math. Intelligencer, {\bf 2}
(1979/1980), no. 2, 62--64.

\bibitem{le}
D.~Leshchiner, {\it Some new identities for $\zeta(k)$,} J. Number
Theory, {\bf 13} (1981), 355--362.


\bibitem{ma}
A.~A.~Markoff, {\it M\'emoir\'e sur la transformation de s\'eries
peu convergentes en s\'eries tres convergentes,} M\'em. de l'Acad.
Imp. Sci. de St. P\'etersbourg, t. XXXVII, No.9 (1890), 18pp.
Available at
http://www.math.mun.ca/\~{}sergey/Research/History/Markov/markov1890.html

\bibitem{mo}
M.~Mohammed, {\it Infinite families of accelerated series for some
classical constants by the Markov-WZ method,}  Discrete
Mathematics and Theoretical Computer Science {\bf 7} (2005), 11--24.

\bibitem{moze}
M.~Mohammed, D.~Zeilberger, {\it The Markov-WZ method,} Electronic
J. Combinatorics {\bf 11} (2004), $\#$R53.

\bibitem{pwz}
M.~Petkov$\check{s}$ek, H.~Wilf, D.~Zeilberger, {\it A=B,} A.~K.~Peters.~Ltd, 1997.


\bibitem{po}
A.~van der Poorten, {\it A proof that Euler
missed\ldots Ap\'ery's proof of the irrationality of $\zeta(3).$ An
informal report,} Math. Intelligencer {\bf 1} (1978/79), no. 4,
195--203.

\bibitem{ri}
T.~Rivoal, {\it Simultaneous generation of Koecher and Almkvist-Granville's
Ap\'ery-like formulae,} Experiment.~Math., {\bf 13} (2004), 503--508.

\bibitem{sl}
L.~J.~Slater, {\it Generalized Hypergeometric Functions,} Cambridge University Press, 1966.

\bibitem{s0}
Z.~H.~Sun, {\it Congruences concerning Bernoulli numbers and Bernoulli polynomials,}
Discrete Appl. Math. {\bf 105} (2000), 193--223.

\bibitem{sun1}
Z.-W.~Sun, {\it Super congruences and Euler numbers,} arXiv:1001.4453 [math.NT]

\bibitem{sun2}
Z.-W.~Sun, {\it Products and sums divisible by central binomial coefficients,}
arXiv:1004.4623 [math.NT]

\bibitem{t1}
R.~Tauraso, {\it More congruences for central binomial coefficients,}  J. Number Theory, {\bf 130} (2010), 2639--2649.


\bibitem{t2}
R.~Tauraso, {\it New harmonic number identities with applications,} 
S\'em. Lothar. Combin. {\bf 63} (2010), Art. B63g, 8 pp.


\bibitem{w}
F.~J.~W.~Whipple, {\it Well-poised series and other generalized hypergeometric series,}
Proc. London Math. Soc. {\bf 25} (1926), no.~2, 525-544.

\bibitem{zeil}
D.~Zeilberger, {\it Closed form (pun intended!),} Contemporary Math. {\bf 143} (1993), 597--607.



\end{thebibliography}
\end{document}